\documentclass[12pt]{article}
\usepackage{amsmath,amsthm,amssymb}
\usepackage{times}
\usepackage{enumerate}
\usepackage{hyperref}

\theoremstyle{definition}
\newtheorem{theorem}{Theorem}[section]

\newtheorem{lemma}[theorem]{Lemma}

\newtheorem{remark}[theorem]{Remark}

\numberwithin{equation}{section}

\usepackage[margin=1.1in]{geometry}
\newcommand{\subjclass}[1]{\bigskip\noindent\emph{2010 Mathematics Subject Classification:}\enspace#1}
\newcommand{\keywords}[1]{\noindent\emph{Keywords:}\enspace#1}

\usepackage{graphicx,graphics,color}
\usepackage{tikz}

\newcommand\bfd{{\mathbf d}}
\newcommand\bfe{{\mathbf e}}
\newcommand\bff{{\mathbf f}}
\newcommand\bfg{{\mathbf g}}

\newcommand\bfn{{\mathbf n}}

\newcommand\bfu{{\mathbf u}}
\newcommand\bfv{{\mathbf v}}
\newcommand\bfw{{\mathbf w}}
\newcommand\bfx{{\mathbf x}}

\newcommand\bfz{{\mathbf z}}
\newcommand\bfA{{\mathbf A}}

\newcommand\bfH{{\mathbf H}}
\newcommand\bfF{{\mathbf F}}

\newcommand\bfK{{\mathbf K}}
\newcommand\bfM{{\mathbf M}}

\newcommand\qon{\quad\hbox{ on }\quad}

\renewcommand{\d}{\text{d}}

\newcommand{\Ga}{\Gamma}

\newcommand{\nbg}{\nabla_{\Gamma}}

\newcommand{\mat}{\partial^{\bullet}}

\newcommand{\diff}{\frac{\d}{\d t}}
\newcommand{\eps}{\varepsilon}

\newcommand{\inv}{^{-1}}

\newcommand{\nb}{\nabla}

\newcommand{\pa}{\partial}
\newcommand{\R}{\mathbb{R}}

\newcommand{\spn}{\textnormal{span}}

\def \u {u_0}

\def \t {(t)}

\def \to {\rightarrow}

\def \u {u}
\newcommand{\phiv}{\varphi^v}
\newcommand{\phin}{\varphi^\n}
\newcommand{\phiH}{\varphi^H}
\newcommand{\phiu}{\varphi^{\u}}

\newcommand{\normM}[1]{\| #1 \|_{\bfM(\xs)}}

\newcommand{\normA}[1]{\| #1 \|_{\bfA(\xs)}}

\newcommand{\normK}[1]{\| #1\|_{\bfK(\xs)}}

\newcommand{\xs}{\bfx^\ast}

\newcommand{\eu}{\bfe_\bfu}

\newcommand{\ex}{\bfe_\bfx}
\newcommand{\en}{\bfe_\bfn}
\newcommand{\eH}{\bfe_\bfH}
\newcommand{\ew}{\bfe_\bfw}

\newcommand{\n}{\nu}

\newcommand{\dof}{N}

\newcommand{\bcl}{\color{black}}
\newcommand{\ecl}{\color{black}}

\newcommand{\bbk}{\color{black}}
\newcommand{\ebk}{\color{black}}

\begin{document}




\title{\vspace{-25pt}
A convergent algorithm for forced mean curvature flow driven by diffusion on the surface}

\author{\normalsize Bal\'{a}zs Kov\'{a}cs \\
\normalsize Mathematisches Institut, Universit\"at T\"{u}bingen, \\
\normalsize Auf der Morgenstelle 10, 72076 T\"{u}bingen, Germany \\
\normalsize kovacs@na.uni-tuebingen.de \\
\\
\normalsize Buyang Li\\
\normalsize Department of Applied Mathematics,\\
\normalsize The Hong Kong Polytechnic University, Hong Kong \\
\normalsize buyang.li@polyu.edu.hk \\
\\
\normalsize Christian Lubich\\
\normalsize Mathematisches Institut, Universit\"at T\"{u}bingen, \\
\normalsize Auf der Morgenstelle 10, 72076 T\"{u}bingen, Germany \\
\normalsize lubich@na.uni-tuebingen.de\vspace{10pt}
}

\date{\normalsize \today}

\maketitle
\vspace{-10pt}


\begin{abstract}
	The evolution of a closed two-dimensional surface driven by both mean curvature flow and a reaction--diffusion process on the surface is formulated as a system that couples the velocity law not only to the surface partial differential equation but also to the evolution equations for the normal vector and the mean curvature on the surface. 
	Two algorithms are considered for the obtained system. Both methods combine surface finite elements for space discretization and linearly implicit backward difference formulae for time integration. Based on our recent results for mean curvature flow, one of the algorithms directly admits a convergence proof for its full discretization in the case of finite elements of polynomial degree at least two and backward difference formulae of orders two to five, with optimal-order error bounds.
	Numerical examples are provided to support and complement the theoretical convergence results (illustrating the convergence behaviour of both algorithms) and demonstrate the effectiveness of the methods in simulating a three-dimensional tumour growth model.

\bigskip

\end{abstract}

\section{Introduction}

\bcl
We consider the numerical approximation of an unknown evolving two-dimensional closed surface $\Gamma(t)$ that is driven by both mean curvature flow and a reaction--diffusion process on the surface, starting from a given smooth initial surface $\Gamma^0$.
 The outer normal velocity $V$ of the surface is determined by the velocity law
\begin{equation}
\label{mcf-v1}
	V=-H+\u,
\end{equation}
where $H$ is the mean curvature of the evolving surface, and where  $\u(x,t)$ ($x\in \Gamma(t)$, $t\in[0,T]$) is the solution of a reaction--diffusion equation on the evolving surface,
\begin{equation}
\label{mcf-diffusion1}
\mat \u + \u \,\nb_{\Ga} \cdot v - \varDelta_{\Ga} \u = F(\u,\nb_{\Ga}\u) ,
\end{equation}
with given initial data $\u^0$. Here, $F: \R \times \R^{3} \to \R$  is a given smooth function, and $v$ is the surface velocity: $v=V\nu$ with $V$ of \eqref{mcf-v1} and the outer normal $\nu$.
Problem \eqref{mcf-v1}--\eqref{mcf-diffusion1} can be viewed as \emph{forced} mean curvature flow driven by the solution of the parabolic equation \eqref{mcf-diffusion1} on the evolving surface.  

While we study the numerical approximation of Problem \eqref{mcf-v1}--\eqref{mcf-diffusion1} with a scalar parabolic equation for notational simplicity, we remark that the numerical method and its convergence properties extend readily to the case of a {\it system} of reaction-diffusion equations 
\eqref{mcf-diffusion1} with solution $\u=(\u_1,\dots,\u_m)$ and the velocity law $V=-H+\alpha_1\u_1+\ldots+\alpha_m\u_m$ with constant real coefficients~$\alpha_i$. We will encounter such a more general problem in our numerical experiments with a tumour growth model.
\ecl

Many practical applications concern mean curvature flow coupled with surface partial differential equations (PDEs), for example tumour growth \cite{CrampinGaffneyMaini1999,CrampinGaffneyMaini2002,CGG,BarreiraElliottMadzvamuse2011,EylesKingStyles2019}; 
surface dissolution \cite{Erlebacheretal2001,EilksElliott2008} (also see \cite[Section~10.4]{DziukElliott_acta}); 
diffusion induced grain boundary motion \cite{FifeCahnElliott2001,DeckelnickElliottStyles2001,Styles2013}. 
 \bbk These models all use a velocity law that is linear in $u$, as in \eqref{mcf-v1} or as in the previous paragraph, except  for diffusion induced grain boundary motion where $V = -H+\u^2$.
\ecl

Numerical approximations to forced mean curvature flow coupled with surface partial differential equations have been considered in some of these papers. 
For \emph{curves}, convergence of numerical methods for such coupled problems of forced curve shortening flow was proved 
in~\cite{PozziStinner_curve,BDS}.  

Numerical approximation to pure mean curvature flow of {\it surfaces} --- i.e.~the case $\u \equiv 0$ in \eqref{mcf-v1} --- was first addressed by Dziuk \cite{Dziuk90},
based on a formulation of mean curvature flow as a formally heat-like equation on a surface. He proposed an evolving surface finite element method in which the moving nodes of the finite element mesh determine the approximate evolving surface. 
 Different surface finite element based methods were proposed by Barrett, Garcke \& N\"urnberg \cite{BGN2008} based on different variational formulations, and by Elliott \& Fritz \cite{Elliott-Fritz-2017} based on DeTurck's trick of reparametrizing the surface.
However, proving convergence of any of these methods has remained an open problem for the mean curvature flow of closed two-dimensional surfaces.

In \cite{MCF} we proved the first convergence result for semi- and full discretizations of mean curvature flow of closed surfaces with evolving surface finite elements. Discretizing the coupled system
for the velocity law together with evolution equations for the normal vector field and mean curvature, we obtained a method with provable error bounds of optimal order.


To our knowledge, no convergence results have yet been proved for forced mean curvature flow of closed surfaces \eqref{mcf-v1}--\eqref{mcf-diffusion1}. For a {\it regularized} version of forced mean curvature flow of closed surfaces, optimal-order convergence results for semi- and full discretizations were obtained in \cite{KLLP2017} and \cite{KL2018}, respectively.

In this paper, we extend the approach and techniques of our previous paper \cite{MCF} to the forced mean curvature flow problem \eqref{mcf-v1}--\eqref{mcf-diffusion1} as a coupled problem together with evolution equations for the normal vector and mean curvature. These evolution equations, as compared with  those  for pure mean curvature flow given in \cite{Huisken1984}, contain additional forcing terms depending on~$\u$. We present two fully discrete evolving finite element algorithms for the obtained coupled system. \bcl The first algorithm discretizes the two terms $\mat \u + \u \,\nb_{\Ga[X]} \cdot v$ separately in the spatial discretization by using the velocity law for $v$ and the approach in \cite{MCF}. 
The second algorithm combines the two terms in the spatial discretization by an idea of \cite{DziukElliott_ESFEM} for treating conservation laws on an evolving surface. \ecl
Both algorithms use evolving surface finite elements for spatial discretization and linearly implicit backward difference formulae for time integration, and 
for both algorithms the moving nodes of a finite element mesh determine the approximate evolving surface. 

\bcl The convergence proof for the forced mean curvature algorithm considered here is a very minor modification compared to the convergence proof for the pure mean curvature algorithm of \cite{MCF}, since that algorithm is already built on coupling evolution equations on the surface to the evolution of the surface. 
The first algorithm can be written in the same matrix--vector form as the method proposed in \cite{MCF} for the mean curvature flow. 
The convergence analysis in \cite{MCF} applies directly to the present algorithm for forced mean curvature flow as well, except for one term which corresponds to the term $\Delta_\Ga u$ in the evolution equation for $H$. The necessary changes to the stability analysis brought about by this term are carried out in detail. \ecl
Under the assumption that the problem admits a sufficiently regular solution, this yields uniform in time, optimal-order $H^1$-norm convergence results for the semi- and full discretizations of forced mean curvature flow when using at least quadratic evolving surface FEM and linearly implicit backward difference formulae of order two to five. 

\bcl 
For the  second algorithm, we indicate how such an optimal-order convergence estimate of the evolving surface finite element semi-discretization can be obtained by combining results of \cite{KLLP2017} and \cite{MCF}, but we do not carry out the details.

For the velocity law $V=-H+g(u)$ with a nonlinear smooth function $g$, we expect that convergence of a direct generalization of the algorithms presented in this paper can be shown with a combination of the techniques of \cite{MCF,KLLP2017,LubichMansour_wave}.
As this would become a nontrivial lengthy extension, it is not worked out here.
\ecl

Finally, we present numerical experiments to support and complement the theoretical results. We present convergence tests for both algorithms, and also present an experiment with the numerical simulation for a tumour growth model, using the parameters in \cite{BarreiraElliottMadzvamuse2011} for the sake of easy comparison.


\section{Evolution equations for mean curvature flow driven by diffusion on the surface}

\subsection{Basic notions and notation}
\label{subsection: basic notions}

We consider the evolving two-dimensional closed surface $\Gamma(t)\subset\R^3$ for times $t\in[0,T]$ as the image
$$
\Ga(t)=  \Ga[X(\cdot,t)] := \{ X(p,t) \,:\, p \in \Ga^0 \},
$$
of a smooth flow map $X:\Ga^0\times [0,T]\to \R^3$ such that $X(\cdot,t)$ is an embedding for every $t$. Here, $\Ga^0$ is a smooth closed initial surface, and $X(p,0)=p$. When the time $t$ is clear from the context, we drop $t$ in the notation and write for short
$$
\Ga[X] = \Ga[X(\cdot,t)].
$$
In view of the subsequent numerical discretization, it is convenient to think of $X(p,t)$ as the position at time $t$ of a moving particle with label $p$, and of $\Ga[X] $ as a collection of such particles. 

The {\it velocity} $v(x,t)\in\R^3$ at a point $x=X(p,t)\in\Ga(t)$  equals
\begin{equation} \label{velocity}
\partial_t X(p,t)= v(X(p,t),t).
\end{equation}
For a known velocity field  $v$,  the position $X(p,t)$ at time $t$ of the particle with label $p$ is obtained by solving the ordinary differential equation \eqref{velocity} from $0$ to $t$ for a fixed $p$.

For a function $w(x,t)$ ($x\in \Gamma(t)$, $0\le t \le T$) we denote the {\it material derivative} as
$$
\mat w(x,t) = \frac \d{\d t} \,w(X(p,t),t) \quad\hbox{ for } \ x=X(p,t).
$$

On any regular surface $\Gamma\subset\R^3$, we denote by $\nabla_{\Ga}w:\Gamma\to\R^3$ the  {\it tangential gradient} of a function $w:\Gamma\to\R$, and in the case of a vector field $f=(f_1,f_2,f_3)^T:\Gamma\to\R^3$, we let
$\nabla_{\Ga}f=
(\nabla_{\Ga}f_1,
\nabla_{\Ga}f_2,
\nabla_{\Ga}f_3)$. We thus use the convention that the gradient of $f$ has the gradient of the components as column vectors. 
We denote by $\nabla_{\Ga} \cdot f$ the {\it surface divergence} of a vector field $f$ on $\Gamma$, 
and by 
$\varDelta_{\Ga} w=\nabla_{\Ga}\cdot \nabla_{\Ga}w$ the {\it Laplace--Beltrami operator} applied to $w$; see the review \cite{DeckelnickDE2005} or \cite[Appendix~A]{Ecker2012} or any textbook on differential geometry for these notions. 

We denote the unit outer normal vector field to $\Gamma$ by $\n:\Gamma\to\R^3$. Its surface gradient contains the (extrinsic) curvature data of the surface $\Gamma$. At every $x\in\Gamma$, the matrix of the extended Weingarten map,
$$
A(x)=\nabla_\Gamma \n(x),
$$ 
is a symmetric $3\times 3$ matrix (see, e.g., \cite[Proposition~20]{Walker2015}). Apart from the eigenvalue $0$ with eigenvector $\n$, its other two eigenvalues are
the principal curvatures $\kappa_1$ and $\kappa_2$. They determine the fundamental quantities
\begin{align}\label{Def-H-A2}
H:={\rm tr}(A)=\kappa_1+\kappa_2, \qquad 
|A|^2 = \kappa_1^2 +\kappa_2^2 ,
\end{align}
where $|A|$ denotes the Frobenius norm of the matrix $A$.
Here, $H$ is called the {\it mean curvature} (as in most of the literature, we do not put a factor 1/2). 

\subsection{Evolution equations for normal vector and mean curvature} 
Forced mean curvature flow driven by diffusion on the surface 
sets the velocity \eqref{velocity} of the surface $\Ga[X]$ to
\begin{equation}
\label{mcf-v}
	v = V\n \quad\ \text{ with the normal velocity } \quad  V=-H+\u ,
\end{equation}
where  $\u$ is the solution of the non-linear reaction--diffusion equation on the surface $\Ga[X]$ with given initial value $\u^0$,
\begin{equation}
\label{mcf-diffusion}
\mat \u + \u \nb_{\Ga[X]} \cdot v - \varDelta_{\Ga[X]} \u = F(\u,\nb_{\Ga[X]}\u)  \qon \Ga[X],
\end{equation}
with a given smooth function $F: \R \times \R^{3} \to \R$.

The geometric quantities $H$ and $\n$ on the right-hand side of \eqref{mcf-v} satisfy the following evolution equations, which are modifications of the evolution equations for pure mean curvature flow (i.e.,~$V=-H$) as derived by Huisken \cite{Huisken1984}.
\begin{lemma}
	\label{lemma:evolution equations for geometric variables}
	For a regular surface $\Ga[X]$ moving under forced mean curvature flow \eqref{mcf-v},  the normal vector and the mean curvature satisfy
	\begin{align}
	\label{Eq_n} 
	\mat \n &= \varDelta_{\Ga[X]} \n + |A|^2 \,\n - \nb_{\Ga[X]} \u , \\
	\label{Eq_H}
	\mat H &= \varDelta_{\Ga[X]} H + |A|^2 H  - \varDelta_{\Ga[X]} \u - |A|^2 \u. 
	\end{align}
\end{lemma}
\begin{proof}
	Using the normal velocity $V$ in the proof of \cite[Lemma~3.3]{Huisken1984}, or see also \cite[Lemma~2.37]{BGN_survey}, the following evolution equation for the normal vector holds: 
	$$
	\mat \n = - \nabla_{\Ga[X]} V . 
	$$
	On any surface $\Gamma$, it holds true that (see \cite[(A.9)]{Ecker2012} or \cite[Proposition~24]{Walker2015})
	$$
	\nabla_{\Ga[X]} H = \varDelta_{\Ga[X]} \n  +|A|^2 \n,
	$$
	which, in combination with $V=-H + \u$ from \eqref{mcf-v}, gives the stated evolution equation for~$\n$.
	
	By revising the proof of \cite[Theorem~3.4 and Corollary~3.5]{Huisken1984}, or see \cite[Lemma~2.39]{BGN_survey}, with the normal velocity $V$ we obtain 
	\begin{equation*}
	\mat H = -\varDelta_{\Ga[X]} V - |A|^2 V ,
	\end{equation*}
	which, again with $V=-H + \u$ from \eqref{mcf-v}, yields the evolution equation for~$H$.
\end{proof}

\subsection{The system of equations used for discretization}
Similarly to \cite{MCF}, collecting the above equations, we have reformulated forced mean curvature flow as the system of semi-linear parabolic equations \eqref{Eq_n}--\eqref{Eq_H} on the surface coupled to the velocity law \eqref{mcf-v} and the surface PDE \eqref{mcf-diffusion}.
The numerical discretization is based on a weak formulation of \eqref{mcf-v}--\eqref{Eq_H}, together with the velocity equation \eqref{velocity}. \bbk For the velocity law \eqref{mcf-v} we use a weak formulation that turns into the standard Ritz projection when restricted to a subspace. The weak formulation reads, with $V=-H+\u$ and $A=\nabla_{\Ga[X]}\nu$, \ebk 
\begin{subequations}
	\label{weak form}
	\begin{align}
	\label{eq:velocity law - weak form}
	& \int_{\Ga[X]} \!\!\!\! \nabla_{\Ga[X]} v \cdot  \nabla_{\Ga[X]} \phiv + \! \int_{\Ga[X]} \!\!\!\!  v \cdot \phiv  
	=   \!\int_{\Ga[X]} \!\!\!\!\! \nabla_{\Ga[X]}(V\n) \cdot \nabla_{\Ga[X]}\phiv + \int_{\Ga[X]} \!\!\!\!\! V\n \cdot \phiv 
	\\
	\label{eq:PDE for normal - weak form}
	& \int_{\Ga[X]} \!\!\!\! \mat \n \cdot \phin + \int_{\Ga[X]} \!\!\!\! \nabla_{\Ga[X]} \n \cdot \nabla_{\Ga[X]} \phin 
	=  \int_{\Ga[X]} \!\!\!\! |A|^2\,   \n\, \cdot \phin
	- \int_{\Ga[X]} \!\!\!\! \nabla_{\Ga[X]}\u \cdot \phin \\
	\label{eq:PDE for mean curvature - weak form}
	& \int_{\Ga[X]} \!\!\!\!\! \mat H \, \phiH \! + \! \int_{\Ga[X]} \!\!\!\!\! \nabla_{\Ga[X]} H \cdot  \nabla_{\Ga[X]} \phiH 
	= \! -\int_{\Ga[X]} \!\!\!\!\! | A|^2 \, V\, \phiH 
	\nonumber \\
	&\ \phantom{\int_{\Ga[X]} \!\!\!\!\! \mat H \, \phiH \! + \! \int_{\Ga[X]} \!\!\!\!\! \nabla_{\Ga[X]} H \cdot  \nabla_{\Ga[X]} \phiH = } + \! \int_{\Ga[X]} \!\!\!\!\! \nabla_{\Ga[X]}\u \cdot \nabla_{\Ga[X]} \phiH ,
	\end{align}
\end{subequations}
\begin{equation}
\label{eq:surface PDE - weak form}
	\diff \int_{\Ga[X]} \u \phiu + \int_{\Ga[X]} \nabla_{\Ga[X]} \u \cdot \nabla_{\Ga[X]} \phiu =  \int_{\Ga[X]} F(\u,\nb_{\Ga[X]}\u) \phiu ,
\end{equation}
for all test functions $\phiv \in H^1(\Ga[X])^3$ and $\phin \in H^1(\Ga[X])^3$, $\phiH \in H^1(\Ga[X])$, and $\phiu \in H^1(\Ga[X])$ with $\mat \phiu = 0$. Here,
we use the Sobolev space 
$H^1(\Ga)=\{ u \in L^2(\Gamma)\,:\, \nabla_\Gamma u \in L^2(\Gamma) \}$. Throughout the paper both the usual Euclidean scalar product for vectors and the Frobenius inner product for matrices (which equals to the Euclidean product using an arbitrary vectorization) are denoted by a dot. 
This system is complemented with the initial data $X^0$, $\n^0$, $H^0$ and~$\u^0$.

An alternative weak formulation of \eqref{eq:surface PDE - weak form}, which is similar to \eqref{eq:PDE for normal - weak form}--\eqref{eq:PDE for mean curvature - weak form},  is based on
\begin{align*}
&\ \int_{\Ga[X]} \mat \u \,\phiu + \int_{\Ga[X]} \nabla_{\Ga[X]} \u \cdot \nabla_{\Ga[X]} \phiu 
=   \int_{\Ga[X]} F(\u,\nb_{\Ga[X]}\u) \phiu - \int_{\Ga[X]} \big( \nb_{\Ga[X]}\cdot v\big) \u  \phiu ,
\end{align*}
for $\phiu \in H^1(\Ga[X])$. \bbk 
Using that $\nbg V \cdot \nu = 0$ and $H = \nb_{\Ga[X]} \cdot \n$ and inserting the velocity law \eqref{mcf-v}, we obtain 
\begin{align*}
\nb_{\Ga[X]} \cdot v = &\ \nb_{\Ga[X]} \cdot (V \n) 
=   (\nb_{\Ga[X]} V) \cdot \n + V (\nb_{\Ga[X]} \cdot \n) 
=  
  V (\nb_{\Ga[X]} \cdot \n) \\
= &\ \bbk  (-H + \u) H .
\end{align*}
This yields a weak formulation of a similar form as \eqref{eq:PDE for normal - weak form} and \eqref{eq:PDE for mean curvature - weak form},
\begin{equation}
\label{eq:surface PDE - alternative weak form}
\begin{aligned}
&\ \int_{\Ga[X]} \mat \u\, \phiu + \int_{\Ga[X]} \nabla_{\Ga[X]} \u \cdot \nabla_{\Ga[X]} \phiu 
=  \int_{\Ga[X]} f(H,\u,\nb_{\Ga[X]}\u) \phiu 
\end{aligned}
\end{equation}
for all $\phiu \in H^1(\Ga[X])$, where we set
$$
f(H,\u,\nb_{\Ga[X]}\u) = F(\u,\nb_{\Ga[X]}\u) -  (-H + \u) Hu.
$$
\ebk

%

\section{Evolving finite element semi-discretization}
\label{section:ESFEM}

\subsection{Evolving surface finite elements}
We formulate the evolving surface finite element (ESFEM) discretization for the velocity law coupled with evolution equations on the evolving surface, following the description in \cite{KLLP2017,MCF}, which is based on \cite{Dziuk88} and \cite{Demlow2009}. \bbk We use a surface approximation consisting of curved elements of polynomial degree $k$ over a flat triangular reference element, which are therefore simply called triangles (even if they are curved), and use \ebk continuous piecewise polynomial basis functions of degree~$k$, as defined in \cite[Section 2.5]{Demlow2009}.

We triangulate the given smooth initial surface $\Gamma^0$ by an admissible family of triangulations $\mathcal{T}_h$ of decreasing maximal element diameter $h$; see \cite{DziukElliott_ESFEM} for the notion of an admissible triangulation, which includes quasi-uniformity and shape regularity. For a momentarily fixed $h$, we denote by $\bfx^0 $  the vector in $\R^{3\dof}$ that collects all nodes $p_j$ $(j=1,\dots,\dof)$ of the initial triangulation. By piecewise polynomial interpolation of degree $k$, the nodal vector defines an approximate surface $\Gamma_h^0$ that interpolates $\Gamma^0$ in the nodes $p_j$ \bbk of the (curved) triangles of $\mathcal{T}_h$. \ebk We will evolve the $j$th node in time \bbk according to an approximation of the ODE \eqref{velocity}, \ebk denoted $x_j(t)$ with $x_j(0)=p_j$, and collect the nodes at time $t$ in a column vector
$$
\bfx(t) \in \R^{3\dof}. 
$$
We just write $\bfx$ for $\bfx(t)$ when the dependence on $t$ is not important.

By piecewise polynomial interpolation on the  plane reference triangle that corresponds to every
curved triangle of the triangulation, the nodal vector $\bfx$ defines a closed surface denoted by $\Gamma_h[\bfx]$. We can then define globally continuous finite element {\it basis functions}
$$
\phi_i[\bfx]:\Gamma_h[\bfx]\to\R, \qquad i=1,\dotsc,\dof,
$$
which have the property that on every triangle their pullback to the reference triangle is polynomial of degree $k$, and which satisfy at the nodes $\phi_i[\bfx](x_j) = \delta_{ij}$ for all $i,j = 1,  \dotsc, \dof .$
These functions span the finite element space on $\Gamma_h[\bfx]$,
\begin{equation*}
S_h[\bfx] = S_h(\Gamma_h[\bfx])=\spn\big\{ \phi_1[\bfx], \phi_2[\bfx], \dotsc, \phi_\dof[\bfx] \big\} .
\end{equation*}
For a finite element function $u_h\in S_h[\bfx]$, the tangential gradient $\nabla_{\Gamma_h[\bfx]}u_h$ is defined piecewise on each element.

The discrete surface at time $t$ is parametrized by the initial discrete surface via the map $X_h(\cdot,t):\Gamma_h^0\to\Gamma_h[\bfx(t)]$ defined by
$$
X_h(p_h,t) = \sum_{j=1}^\dof x_j(t) \, \phi_j[\bfx(0)](p_h), \qquad p_h \in \Gamma_h^0,
$$
which has the properties that $X_h(p_j,t)=x_j(t)$ for $j=1,\dots,\dof$, that  $X_h(p_h,0) = p_h$ for all $p_h\in\Gamma_h^0$, and
$$
\Gamma_h[\bfx(t)]=\Gamma[X_h(\cdot,t)] = \{ X_h(p_h,t) \,:\, p_h \in \Ga_h^0 \}. 
$$

The {\it discrete velocity} $v_h(x,t)\in\R^3$ at a point $x=X_h(p_h,t) \in \Gamma[X_h(\cdot,t)]$ is given by
$$
\partial_t X_h(p_h,t) = v_h(X_h(p_h,t),t).
$$
In view of the transport property of the basis functions  \cite{DziukElliott_ESFEM},
$
\tfrac\d{\d t} \big( \phi_j[\bfx(t)](X_h(p_h,t)) \big) =0 ,
$
the discrete velocity equals, for $x \in \Gamma_h[\bfx(t)]$,
$$
v_h(x,t) = \sum_{j=1}^\dof v_j(t) \, \phi_j[\bfx(t)](x) \qquad \hbox{with } \ v_j(t)=\dot x_j(t),
$$
where the dot denotes the time derivative $\d/\d t$. 
Hence, the discrete velocity $v_h(\cdot,t)$ is in the finite element space $S_h[\bfx(t)]$, with nodal vector $\bfv(t)=\dot\bfx(t)$.
%
%

The {\it discrete material derivative} of a finite element function $u_h(x,t)$ with nodal values $u_j(t)$ is
$$
\mat_h u_h(x,t) = \frac{\d}{\d t} u_h(X_h(p_h,t)) = \sum_{j=1}^\dof \dot u_j(t)  \phi_j[\bfx(t)](x)  \quad\text{at}\quad x=X_h(p_h,t).
$$


\subsection{ESFEM spatial semi-discretizations}
\label{subsection:semi-discretization}

Now we will describe the semi-discretization of the coupled system using both formulations of the surface PDE.

The finite element spatial semi-discretization of the weak coupled parabolic system \eqref{weak form} and \eqref{eq:surface PDE - alternative weak form} reads as follows: Find the unknown nodal vector $\bfx(t)\in \R^{3\dof}$ and the unknown finite element functions $v_h(\cdot,t)\in S_h[\bfx(t)]^3$ and $\n_h(\cdot,t)\in S_h[\bfx(t)]^3$, $H_h(\cdot,t)\in S_h[\bfx(t)]$, and $\u_h(\cdot,t)\in S_h[\bfx(t)]$ such that, by denoting $\alpha_h^2 = | \nb_{\Ga_h[\bfx]} \n_h |^2$ and $V_h = -H_h + \u_h$, 
\begin{subequations}
	\label{eq:semidiscrete weak form}
	\begin{align}
		\label{eq:semidiscrete v}
		& \int_{\Ga_h[\bfx]} \!\!\!\!\! \nb_{\Ga_h[\bfx]} v_h \cdot \nb_{\Ga_h[\bfx]} \phiv_h +\int_{\Ga_h[\bfx]} \!\!\!\!\! v_h \cdot \phiv_h 
		=  \int_{\Ga_h[\bfx]} \!\!\!\!\! \nabla_{\Ga_h[\bfx]}(V_h\n_h) \cdot \nabla_{\Ga_h[\bfx]}\phiv_h + \int_{\Ga_h[\bfx]} \!\!\!\! V_h \n_h \cdot \phiv_h 
		\\ 
		\label{eq:semidiscrete nu}
		&\int_{\Ga_h[\bfx]} \!\!\!\!\! \mat_h \n_h \cdot \phin_h +\int_{\Ga_h[\bfx]}\!\!\!\! \nb_{\Ga_h[\bfx]} \n_h \cdot \nb_{\Ga_h[\bfx]} \phin_h 
		= \int_{\Ga_h[\bfx]} \!\!\!\!\! \alpha_h^2  \, \n_h \cdot \phin_h 
		- \int_{\Ga_h[\bfx]} \!\!\!\!\! \nb_{\Ga_h[\bfx]} \u_h \cdot \phin_h\\
		\label{eq:semidiscrete H}
		&\int_{\Ga_h[\bfx]} \!\!\!\!\! \mat_h H_h \, \phiH_h + \int_{\Ga_h[\bfx]} \!\!\!\!\!  \nb_{\Ga_h[\bfx]} H_h \cdot \nb_{\Ga_h[\bfx]} \phiH_h  
		= -\int_{\Ga_h[\bfx]}\!\!\!\! \alpha_h^2 \, V_h \, \phiH_h 
		\nonumber\\
		&\hskip 7.5cm + \int_{\Ga_h[\bfx]} \!\!\!\!\! \nb_{\Ga_h[\bfx]} \u_h \cdot \nb_{\Ga_h[\bfx]} \phiH_h		
	\end{align}
\end{subequations}
and
\begin{equation}
\label{eq:semidiscrete surface PDE - alternative}
\begin{aligned}
&\ \int_{\Ga_h[\bfx]} \!\!\!\! \mat_h \u_h \phiu_h + \int_{\Ga_h[\bfx]} \!\!\!\!\! \nabla_{\Ga[X]} \u_h \cdot \nabla_{\Ga_h[\bfx]} \phiu_h 
= 
\int_{\Ga_h[\bfx]} \!\!\!\!  f(H_h,\u_h,\!\nb_{\Ga_h[\bfx]}\u_h) \phiu_h ,
\end{aligned}
\end{equation}
for all $\phiv_h\in S_h[\bfx(t)]^3$, $\phin_h\in S_h[\bfx(t)]^3$, $\phiH_h\in S_h[\bfx(t)]$, and   $\phiu_h\in S_h[\bfx(t)]$  with the surface $\Gamma_h[\bfx(t)]=\Gamma[X_h(\cdot,t)] $ given by the differential equation
\begin{equation}\label{xh}
\partial_t X_h(p_h,t) = v_h(X_h(p_h,t),t), \qquad p_h\in\Ga_h^0.
\end{equation}
The initial values for the nodal vector $\bfx$ are taken as the positions of the nodes of the triangulation of the given initial surface $\Gamma^0$.
The initial data for $\n_h$, $H_h$ and $\u_h$ are determined by Lagrange interpolation of $\n^0$, $H^0$ and $\u^0$, respectively. 

Alternatively, the finite element spatial semi-discretization of the weak coupled parabolic system \eqref{weak form} and \eqref{eq:surface PDE - weak form} determines the same unknown functions, but, instead of \eqref{eq:semidiscrete surface PDE - alternative}, the equations \eqref{eq:semidiscrete weak form} and the ODE \eqref{xh} are coupled to
\begin{equation}
\label{eq:semidiscrete surface PDE}
	\diff \int_{\Ga_h[\bfx]}\!\!\!\! \u_h \phiu_h + \int_{\Ga_h[\bfx]}\!\! \nabla_{\Ga_h[\bfx]} \u_h \cdot \nabla_{\Ga_h[\bfx]} \phiu_h =  \int_{\Ga_h[\bfx]}\!\! F(\u_h,\nb_{\Ga_h[\bfx]}\u_h) \phiu_h
\end{equation}
for all $\phiu_h\in S_h[\bfx(t)]$ \bbk with $\mat_h \phiu_h = 0$. \ebk

In the above approaches, the discretization of the evolution equations for $\n$, $H$ and $\u$ is done in the usual way of evolving surface finite elements. The velocity law \eqref{mcf-v} is enforced by a Ritz projection to the finite element space on $\Gamma_h[\bfx]$. 
Note that the finite element functions $\n_h$ and $H_h$ are \emph{not} the normal vector and the mean curvature of the discrete surface $\Gamma_h[\bfx(t)]$.

\subsection{Matrix--vector formulation}
\label{subsection:DAE}

We collect the nodal values in column vectors  $\bfv=(v_j) \in \R^{3N}$, $\bfn=(\n_j) \in \R^{3N}$, $\bfH=(H_j)\in\R^N$ and $\bfw=(\u_j) \in\R^N$.
We define the surface-dependent mass matrix $\bfM(\bfx)$ and stiffness matrix $\bfA(\bfx)$ 
on the surface determined by the nodal vector $\bfx$:
\begin{equation*}
\begin{aligned}
\bfM(\bfx)\vert_{ij} =&\ \int_{\Ga_h[\bfx]} \! \phi_i[\bfx] \phi_j[\bfx] , \\
\bfA(\bfx)\vert_{ij} =&\ \int_{\Ga_h[\bfx]} \! \nb_{\Ga_h[\bfx]} \phi_i[\bfx] \cdot \nb_{\Ga_h[\bfx]} \phi_j[\bfx] , 
\end{aligned}
\qquad i,j = 1,  \dotsc,\dof ,
\end{equation*}
with the finite element nodal basis functions $\phi_j[\bfx] \in S_h[\bfx]$.
\bcl
We further let, for an arbitrary dimension $d$ (with the identity matrices $I_d \in \R^{d \times d}$),
$$
\bfM^{[d]}(\bfx)= I_d \otimes \bfM(\bfx), \qquad
\bfA^{[d]}(\bfx)= I_d \otimes \bfA(\bfx), \qquad
\bfK^{[d]}(\bfx) = I_d \otimes \bigl( \bfM(\bfx) + \bfA(\bfx) \bigr).
$$
When no confusion can arise, we write $\bfM(\bfx)$ for $\bfM^{[d]}(\bfx)$, $\bfA(\bfx)$ for $\bfA^{[d]}(\bfx)$, and $\bfK(\bfx)$ for $\bfK^{[d]}(\bfx)$.

We define nonlinear functions $\bff(\bfx,\bfn,\bfH,\bfu)\in\R^{5N}$ and $\bfg(\bfx,\bfn,\bfH,\bfu)\in\R^{3N}$, where
$$
\bff(\bfx,\bfn,\bfH,\bfu)
=
\left(\begin{array}{c}
\bff_\nu(\bfx,\bfn,\bfH,\bfu)  \\ 
\bff_H(\bfx,\bfn,\bfH,\bfu)  \\
\bff_u(\bfx,\bfH,\bfu)  
\end{array}\right)
$$
with $\bff_\nu(\bfx,\bfn,\bfH,\bfu)\in \R^{3N}$, $\bff_H(\bfx,\bfn,\bfH,\bfu)\in \R^{N}$ and $\bff_u(\bfx,\bfn,\bfH,\bfu)\in \R^{N}$.
These functions are
given as follows, with the notations $\alpha_h^2 = | \nb_{\Ga_h[\bfx]} \n_h |^2$ and $V_h = -H_h + \u_h$,
\begin{equation*}
\begin{aligned}
\bff_\nu(\bfx,\bfn,\bfH,\bfu)\vert_{j+(\ell-1)N} = &\ \int_{\Ga_h[\bfx]} \!\!\!\! \alpha_h^2 \,(\n_h)_\ell \,  \, \phi_j[\bfx] - \int_{\Ga_h[\bfx]}\!\! \Big(\nb_{\Ga_h[\bfx]} \u_h \Big)_\ell \cdot \phi_j[\bfx], 
\\
\bff_H(\bfx,\bfn,\bfH,\bfu)\vert_j = &\ -\int_{\Ga_h[\bfx]}\!\!\!\! \alpha_h^2 \, V_h \, \phi_j[\bfx] 
+ \int_{\Ga_h[\bfx]}\!\!\!\! \nb_{\Ga_h[\bfx]} \u_h \cdot \nb_{\Ga_h[\bfx]} \phi_j[\bfx] ,
\\
\bff_u(\bfx,\bfH,\bfu)\vert_j = &\ \int_{\Ga_h[\bfx]}\!\!\!\! f(H_h,\u_h,\nb_{\Ga_h[\bfx]}\u_h) \, \phi_j[\bfx] ;
\\
\bfg(\bfx,\bfn,\bfH,\bfu)\vert_{j+(\ell-1)N} = &\ \int_{\Ga_h[\bfx]} \!\!\!\! V_h (\n_h  )_\ell \, \phi_j[\bfx] +
 \int_{\Ga_h[\bfx]} \!\!\!\! \nb_{\Ga_h[\bfx]} (V_h (\n_h  )_\ell)\cdot \nb_{\Ga_h[\bfx]} \phi_j[\bfx] ,
\end{aligned}
\end{equation*}
for $j = 1, \dotsc, N$ and $\ell=1,2,3$. 
We abbreviate
$$
\bfw=\left(\begin{array}{c}
\bfn\\ 
\bfH\\
\bfu
\end{array}\right) \in \R^{5N} .
$$
Equations \eqref{eq:semidiscrete weak form} and \eqref{eq:semidiscrete surface PDE - alternative} with \eqref{xh}
can then be written in the matrix--vector formulation
\begin{equation}
\label{eq:matrix--vector form 2}
\begin{aligned}
\bfK^{[3]}(\bfx)\bfv = &\ \bfg(\bfx,\bfw) , \\
\bfM^{[5]}(\bfx)\dot\bfw + \bfA^{[5]}(\bfx)\bfw = &\ \bff(\bfx,\bfw) ,\\
\dot\bfx = &\ \bfv .
\end{aligned}
\end{equation}
The system \eqref{eq:matrix--vector form 2} for \emph{forced} mean curvature flow is formally the same as the matrix--vector form of the coupled system for \emph{non-forced} mean curvature flow derived in \cite{MCF}, cf.~(3.4)--(3.5) therein, with 
$\bfw=(\bfn; \bfH;\bfu)\in\R^{5N}$ in the role of $\bfu=(\bfn; \bfH)\in\R^{4N}$ of \cite{MCF}. The nonlinearity $\bff(\bfx,\bfw)$ is built up from integrals of the same type as $\bff(\bfx,\bfu)$ in  \cite {MCF}, with the only exception of the second term in $\bff_H$, whose entries contain the tangential gradient of the basis functions and which in total can be written as $\bfA(\bfx)\bfu$. 
This term stems from the term $-\Delta_{\Ga[X]}u$ in the evolution equation for $H$ in Lemma~\ref{lemma:evolution equations for geometric variables}.
The function $\bfg$ is defined in the same way as $\bfg$ in \cite{MCF}, just with $V_h=-H_h+u_h$ in place of $-H_h$.

\begin{remark}
For the alternative system of equations \eqref{eq:semidiscrete weak form} and \eqref{eq:semidiscrete surface PDE} with \eqref{xh} we denote
$$
\bfz=\left(\begin{array}{c}
\bfn\\ 
\bfH
\end{array}\right) \in \R^{4N} ,\qquad 
\bff(\bfx,\bfz,\bfu) = \left(\begin{array}{c}
\bff_\nu(\bfx,\bfn,\bfH,\bfu)  \\ 
\bff_H(\bfx,\bfn,\bfH,\bfu)
\end{array}\right) \in \R^{4N}
$$ 
and introduce
$$
\bfF(\bfx,\bfu)\vert_j =  \int_{\Ga_h[\bfx]}\!\!\!\! F(\u_h,\nb_{\Ga_h[\bfx]}\u_h) \, \phi_j[\bfx].
$$
Equations \eqref{eq:semidiscrete weak form} and \eqref{eq:semidiscrete surface PDE} with \eqref{xh} can then be written in the following matrix--vector form: 
\begin{equation}
\label{eq:matrix--vector form 1}
\begin{aligned}
\bfK^{[3]}(\bfx)\bfv = &\ \bfg(\bfx,\bfz,\bfu) , 
\\
\bfM^{[4]}(\bfx)\dot\bfz+\bfA^{[4]}(\bfx)\bfz = &\ \bff(\bfx,\bfz,\bfu) ,  
\\
\diff \Big( \bfM(\bfx)\bfu \Big) + \bfA(\bfx)\bfu = &\ \bfF(\bfx,\bfu) ,\\
\dot\bfx = &\ \bfv.
\end{aligned}
\end{equation}
\end{remark}
\ecl

\subsection{Lifts} \label{subsec:lifts}

As in \cite{KLLP2017} and \cite[Section~3.4]{MCF}, we compare functions on the {\it exact surface} $\Gamma[X(\cdot,t)]$ with functions on the {\it discrete surface} $\Gamma_h[\bfx(t)]$, via functions on the {\it interpolated surface} $\Gamma_h[\xs(t)]$, where
$\xs(t)$ denotes the nodal vector collecting the grid points $x_j^*(t)=X(p_j,t)$ on the exact surface, where $p_j$ are the nodes of the discrete initial triangulation $\Ga_h^0$.

Any finite element function $w_h$ on the discrete surface, with nodal values $w_j$, is associated with a finite element function $\widehat w_h$ on the interpolated surface $\Gamma_h^*$ with the exact same nodal values. 
This can be further lifted to a function on the exact surface by using the \emph{lift operator} $l$, mapping a function on the interpolated surface $\Gamma_h^*$ to a function on the exact surface $\Gamma$, provided that they are sufficiently close, see \cite{Dziuk88,Demlow2009}. 

Then the composed lift $L$ maps finite element functions on the discrete surface $\Gamma_h[\bfx(t)]$ to functions on the exact surface $\Gamma[X(\cdot,t)]$ via the interpolated surface $\Gamma_h[\xs(t)]$. This is denoted by 
$$
w_h^L = (\widehat w_h)^l.
$$

\section{Convergence of the semi-discretization}
\label{section: main result}

We are now in the position to formulate the first main result of this paper, which yields optimal-order error bounds for the finite element semi-discretization (using finite elements of polynomial degree $k \geq 2$) \eqref{eq:semidiscrete weak form}, and \eqref{eq:semidiscrete surface PDE} or \eqref{eq:semidiscrete surface PDE - alternative}, with \eqref{xh} of the system for forced mean curvature equations \eqref{weak form}, and one of the weak formulations \eqref{eq:surface PDE - weak form} or \eqref{eq:surface PDE - alternative weak form} for the surface PDE, with the ODE \eqref{velocity} for the positions. 
We introduce the notation
$$
	x_h^L(x,t) =  X_h^L(p,t) \in \Gamma_h[\bfx(t)] \qquad\hbox{for}\quad x=X(p,t) \in \Gamma[X(\cdot,t)].
$$

\begin{theorem}
\label{MainTHM} 
	For the coupled forced mean curvature flow problem \eqref{weak form} and 
	\eqref{eq:surface PDE - alternative weak form} with a smooth function $F$, taken together with the velocity equation \eqref{velocity}, 	we consider the space discretization \eqref{eq:semidiscrete weak form}--\eqref{xh} 
	(or equivalently \eqref{eq:matrix--vector form 2} in matrix--vector form) 
with evolving surface finite elements of polynomial degree $k\ge 2$. 
	Suppose that the problem admits an exact solution $(X,v,\n,H,\u)$ that is sufficiently regular on the time interval $t\in[0,T]$, and that the flow map $X(\cdot,t)$ is non-degenerate so that $\Gamma(t)=\Gamma[X(\cdot,t)]$ is a regular surface on the time interval $t\in[0,T]$. 
	
	Then, there exists a constant $h_0 > 0$ such that for all mesh sizes $h \leq h_0$ the following error bounds for the lifts of the discrete position, velocity, normal vector and mean curvature hold over the exact surface $\Ga(t)$ for $0 \leq t \leq T$:
	\begin{align*}
	\|x_h^L(\cdot,t) - \mathrm{id}_{\Ga(t)}\|_{H^1(\Ga(t))^3} \leq &\ Ch^k, \\
	\|v_h^L(\cdot,t) - v(\cdot,t)\|_{H^1(\Ga(t))^3} \leq &\ C h^k, \\
	\|\n_h^L(\cdot,t) - \n(\cdot,t)\|_{H^1(\Ga(t))^3} \leq &\ C h^k, \\
	\|H_h^L(\cdot,t) - H(\cdot,t)\|_{H^1(\Ga(t))} \leq &\ C h^k, \\
	\|\u_h^L(\cdot,t) - \u(\cdot,t)\|_{H^1(\Ga(t))} \leq &\ C h^k, 
	\intertext{and also}
	\|X_h^l(\cdot,t) - X(\cdot,t)\|_{H^1(\Ga_0)^3} \leq &\ Ch^k ,
	\end{align*}
	where the constant $C$ is independent of $h$ and $t$, but depends on bounds of higher derivatives of the solution $(X,v,\n,H,\u)$ of the forced mean curvature flow and on the length $T$ of the time interval.
\end{theorem}

Sufficient regularity assumptions are the following: with bounds that are uniform in $t\in[0,T]$, we assume $X(\cdot,t) \in  H^{k+1}(\Ga^0)^3$
and for $w=(\nu,H,\u)$ we assume $\ w(\cdot,t), \mat w(\cdot,t) \in W^{k+1,\infty}(\Ga(t))^5$. 

\bcl \medskip
Under these strong regularity conditions on the solution, we only require local Lipschitz continuity of the function $F$ in \eqref{mcf-diffusion1}. This condition is, of course, not sufficient to ensure the existence of even just a weak solution. The point here is that we 
restrict our attention to cases where a sufficiently regular solution exists, which we can then approximate with optimal order under weak conditions on~$F$. The regularity theory of Problem~\eqref{mcf-v1}--\eqref{mcf-diffusion1} is, however, outside the scope of this paper.

\bcl
\medskip
The remarks made after the convergence result in \cite{MCF} apply also here. In particular, it is explained that the admissibility of the triangulation over the whole time interval $[0,T]$ is preserved for sufficiently fine grids, provided the exact surface is sufficiently regular. 
\ecl
%

\bcl 
\medskip
\begin{proof}
The proof 
reduces in essence to the proof of Theorem~4.1 in~\cite{MCF}, since the matrix--vector formulation \eqref{eq:matrix--vector form 2} is of precisely the same form as  the matrix--vector formulation of \cite{MCF}, formulas (3.4)--(3.5) therein, with the same mass and stiffness matrices and with nonlinear functions given as integrals over products of smooth pointwise nonlinearities and finite element basis functions (and with $\bfw$ in the role of $\bfu$ of \cite{MCF}). The proof of the stability bounds of \cite[Proposition~7.1]{MCF} uses energy estimates (testing with the time derivative of the error) on the equations of the matrix--vector formulation to bound errors in terms of defects in \eqref{eq:matrix--vector form 2} 
in the appropriate norms. These stability bounds apply immediately to \eqref{eq:matrix--vector form 2} with the same proof, except for one subtle point: Because of the term $-\Delta_{\Ga[X]} \u$ in the evolution equation for $H$ in Lemma~\ref{lemma:evolution equations for geometric variables}, which translates into the second term $\bfA(\bfx)\bfu$ in $\bff_H(\bfx,\bfw)$ in the matrix--vector formulation, the bound for the nonlinearity in part (v) of the proof of  Proposition~7.1 in  \cite{MCF} needs to be changed. This is a very local modification to the proof. No other part of the stability proof is affected.

To explain and resolve this local difficulty, we must assume that the reader has acquired some familiarity with Section~7 of \cite{MCF}. We use the same notation $\bfe_\bfw = \bfw-\bfw^*$ etc.~for the error vectors and note that
$\bfe_\bfw =(\en;\eH;\eu)$ now is in the role of $\bfe_\bfu=(\en;\eH)$ of \cite{MCF}. 
Because of the  extra term  $\bfA(\bfx)\bfu$ in $\bff_H(\bfx,\bfw)$,  the same argument as in part (v) of the proof of  Proposition~7.1 in  \cite{MCF} yields  only  a modified bound
$$
{\dot \bfe}_\bfw^T \big(\bff(\bfx,\bfw) - \bff(\xs,\bfw^*)\big) \le  
 c  \normK{{\dot \bfe}_\bfw} \Bigl( \normK{\bfe_\bfw} + \normA{\ex} \Bigr),
$$
whereas in \cite{MCF} only the weaker norm $\normM{{\dot \bfe}_\bfw}$ appears on the right-hand side. This modified estimate is not sufficient for the further course of the proof. 

It can be circumvented as follows. We write the error vector as $\ew=(\en;\eH;\eu)$ and take the inner product of ${\dot \bfe}_\bfH$ with 
$\big(\bff_H(\bfx,\bfw) - \bff_H(\xs,\bfw^*)\big)$. We note that
$$
\bff_H(\bfx,\bfw) = \widetilde \bff_H(\bfx,\bfw) +\bfA(\bfx)\bfu, 
$$
where $\widetilde \bff_H$ is a nonlinearity of the same type as those studied in \cite{MCF}, and so we have the following bound as in part (v) of the proof of  Proposition~7.1 in \cite{MCF},
$$
{\dot \bfe}_\bfH^T \big(\widetilde \bff_H(\bfx,\bfw) -  \widetilde \bff_H(\xs,\bfw^*)\big) \le c  \normM{{\dot \bfe}_\bfw} \Bigl( \normK{\bfe_\bfw} + \normA{\ex} \Bigr).
$$
For the solution $\bfx(t)$ of \eqref{eq:matrix--vector form 2} we have 
$$
\bfA(\bfx)\bfu = -\bfM(\bfx)\dot \bfu + \bff_u(\bfx,\bfw)
$$
and for the nodal vector $\bfu^*(t)$ of the Ritz projection of the exact solution $u(\cdot,t)$ and the nodal vector $\bfx^*(t)$ of the exact positions we have, with a defect $\bfd_\bfu(t)$,
$$
\bfA(\bfx^*)\bfu^* = -\bfM(\bfx^*)\dot \bfu^* + \bff_u(\bfx^*,\bfw^*) + \bfM(\bfx^*)\bfd_\bfu.
$$
So we can write
\begin{align*}
{\dot \bfe}_\bfH^T \bigl( \bfA(\bfx)\bfu  - \bfA(\bfx^*)\bfu^* \bigr) &= -{\dot \bfe}_\bfH^T \bfM(\bfx)\dot \bfe_\bfu -
{\dot \bfe}_\bfH^T \bigl(  \bfM(\bfx) -  \bfM(\bfx^*) \bigr) \bfu^* \\
&\quad +\bff_u(\bfx,\bfw)  - \bff_u(\bfx^*,\bfw^*) - {\dot \bfe}_\bfH^T \bfM(\bfx^*)\bfd_\bfu.
\end{align*}
By the same estimates as used repeatedly in the proof of  Proposition~7.1 in  \cite{MCF}, this yields
\begin{align*}
{\dot \bfe}_\bfH^T \bigl( \bfA(\bfx)\bfu  - \bfA(\bfx^*)\bfu^* \bigr) \le &\
 \|{\dot \bfe}_\bfH  \|_{\bfM(\bfx)}  \|{\dot \bfe}_\bfu  \|_{\bfM(\bfx)} + c  \|{\dot \bfe}_\bfH  \|_{\bfM(\bfx^*)} \| {\bfe}_\bfx  \|_{\bfA(\bfx^*)}  \\
 &\  + c  \normM{{\dot \bfe}_\bfH} \bigl( \normK{\bfe_\bfw} + \normA{\ex} \bigr)+
 \|{\dot \bfe}_\bfH  \|_{\bfM(\bfx^*)}  \|\bfd_\bfu \|_{\bfM(\bfx^*)}.
 \end{align*}
 We now fix a small $\rho>0$ and use the scaled norm, for ${\dot \bfe}_\bfw=({\dot \bfe}_\bfn;{\dot \bfe}_\bfH;{\dot \bfe}_\bfu)$,
 $$
 \|{\dot \bfe}_\bfw \|_{\bfM(\bfx)}^2 =  \|{\dot \bfe}_\bfn \|_{\bfM(\bfx)}^2 +  \|{\dot \bfe}_\bfH \|_{\bfM(\bfx)}^2 +  \omega^2 \|{\dot \bfe}_\bfu \|_{\bfM(\bfx)}^2 
 $$
 with a large weight $\omega$. If $\omega\ge 1/(2\rho)$, then we have
 $$
  \|{\dot \bfe}_\bfH  \|_{\bfM(\bfx)}  \|{\dot \bfe}_\bfu  \|_{\bfM(\bfx)}  \le \rho  \|{\dot \bfe}_\bfw  \|_{\bfM(\bfx)} ^2.
 $$
 Altogether, this yields the bound
 $$
{\dot \bfe}_\bfw^T \big(\bff(\bfx,\bfw) - \bff(\xs,\bfw^*)\big) \le \rho \|{\dot \bfe}_\bfw  \|_{\bfM(\bfx)} ^2 + c  \normM{{\dot \bfe}_\bfw} \Bigl( \normK{\bfe_\bfw} + \normA{\ex}  +   \normM{{ \bfd}_\bfu}\Bigr).
$$
With this bound, the further parts of the stability proof remain unchanged.

Since the additional terms in \eqref{weak form} and \eqref{eq:surface PDE - alternative weak form} to those in the evolution equations of pure mean curvature flow in \cite{MCF} do not present additional difficulties in the consistency error analysis,  the same bounds for the consistency errors in $(X,v,H,\n,u)$ are obtained as for $(X,v,H,\n)$ in \cite[Proposition~8.1]{MCF}. Furthermore, the combination of the stability bounds and the consistency error bounds to yield optimal-order $H^1$ error bounds is verbatim the same as in \cite[Section~9]{MCF}.
\end{proof}

\begin{remark}
\label{remark:semi-discrete convergence}
	For the semi-discretization \eqref{eq:matrix--vector form 1} a convergence proof can be obtained by combining the convergence proofs of our previous works \cite{KLLP2017} and \cite{MCF}. The stability of the scheme is obtained by combining the results of \cite[Proposition~6.1]{KLLP2017} (in particular part (A)) for the surface PDE, and of \cite[Proposition~7.1]{MCF} for the velocity law and for the geometric quantities, and further using the same modification for the extra term  $\bfA(\bfx)\bfu$ as in the proof above.
	As this extension does not require any new ideas beyond   \cite{KLLP2017} and \cite{MCF}, we do not present the lengthy but straightforward details. Since there are no additional difficulties in bounding the consistency errors, together with the stability bounds we then obtain the same error bounds as in Theorem~\ref{MainTHM}. This is in agreement with the results of numerical experiments presented in Section~\ref{section:numerics}.
\end{remark}
\ecl

\section{Linearly implicit full discretization}

For the time discretization of the system of ordinary differential equations of Section~\ref{subsection:DAE} we use a $q$-step linearly implicit backward difference formula (BDF) with $q \leq 5$. For a step size $\tau>0$, and with $t_n = n \tau \leq T$, let us introduce, for $n \geq q$,
\begin{alignat}{3}
\label{eq:BDF derivative def}
\text{the discrete time derivative} &\qquad& \dot\bfu^n = &\ \frac{1}{\tau} \sum_{j=0}^q \delta_j \bfu^{n-j} , 
\qquad \quad \text{and} \\
\label{eq:extrapolation def}
\text{the extrapolated value} &\qquad& \widetilde \bfu^n = &\ \sum_{j=0}^{q-1} \gamma_j \bfu^{n-1-j}  ,
\end{alignat}
where the coefficients are given by $\delta(\zeta)=\sum_{j=0}^q \delta_j \zeta^j=\sum_{\ell=1}^q \frac{1}{\ell}(1-\zeta)^\ell$ and $\gamma(\zeta) = \sum_{j=0}^{q-1} \gamma_j \zeta^j = (1 - (1-\zeta)^q)/\zeta$, respectively.

We determine the approximations $\bfx^n$ to $\bfx(t_n)$, $\bfv^n$ to $\bfv(t_n)$, and $\bfw^n$ to $\bfw(t_n)$ or $\bfz^n$ to $\bfz(t_n)$ and $\bfu^n$ to $\bfu(t_n)$ (only if not already collected into $\bfw^n$) by the linearly implicit BDF discretization of both systems \eqref{eq:matrix--vector form 2} and \eqref{eq:matrix--vector form 1}.

For \eqref{eq:matrix--vector form 2} we obtain
\begin{equation}
\label{BDF 2}
\begin{aligned}
\bfK(\widetilde \bfx^n) \bfv^n &= \bfg(\widetilde \bfx^n,\widetilde \bfw^n) , \\
\bfM(\widetilde \bfx^n) \dot\bfw^{n} + \bfA(\widetilde \bfx^n) \bfw^n &= \bff(\widetilde \bfx^n,\widetilde \bfw^n) ,  \\
\dot\bfx^{n} &= \bfv^n .
\end{aligned}
\end{equation}

For \eqref{eq:matrix--vector form 1} we obtain
\begin{equation}
\label{BDF 1}
\begin{aligned}
\bfK(\widetilde \bfx^n) \bfv^n &= \bfg(\widetilde \bfx^n,\widetilde \bfz^n,\widetilde \bfu^n) , \\
\bfM(\widetilde \bfx^n) \dot\bfz^{n} + \bfA(\widetilde \bfx^n) \bfz^n &= \bff(\widetilde \bfx^n,\widetilde \bfz^n,\widetilde \bfu^n) ,  \\
\frac{1}{\tau} \sum_{j=0}^q \delta_j \bfM(\widetilde \bfx^{n-j}) \bfu^{n-j} + \bfA(\widetilde \bfx^n) \bfu^n &= \bfF(\widetilde \bfx^n,\widetilde \bfu^n) ,  \\
\dot\bfx^{n} &= \bfv^n .
\end{aligned}
\end{equation}

The starting values $\bfx^i$ and $\bfw^i$, or, in case of \eqref{BDF 1}, $\bfz^i$ and $\bfu^i$, for $i=0,\dotsc,q-1$, are assumed to be given. They can be precomputed using either a lower order method with smaller step sizes or an implicit Runge--Kutta method.

The classical BDF method is known to be $A(\theta)$-stable for some $\theta>0$ for $q\leq6$ and to have order $q$; see \cite[Chapter~V]{HairerWannerII}.
This order is retained by the linearly implicit variant using the above coefficients $\gamma_j$; 
cf.~\cite{AkrivisLubich_quasilinBDF}.


From the vectors $\bfx^n =(x_j^n)$, $\bfv^n = (v_j^n)$, and $\bfw^n=(w_j^n)$ with $w_j^n=(\n_j^n,H_j^n,\u_j^n)\in\R^3 \times \R \times \R$ for the first method and 
$\bfz^n=(z_j^n)$ with $z_j^n=(\n_j^n,H_j^n)\in\R^3\times \R$ and $\bfu^n = (u_j^n)$ for the second method,  we obtain approximations to their respective variables as finite element functions whose nodal values are collected in these vectors.
%

\section{Convergence of the full discretization}

We are now in the position to formulate the second main result of this paper, which yields optimal-order error bounds for the combined ESFEM--BDF full discretizations \eqref{BDF 2} of the forced mean curvature flow problem \eqref{weak form} coupled to the weak form \eqref{eq:surface PDE - alternative weak form} of the surface PDE, with \eqref{velocity}, for finite elements of polynomial degree $k \geq 2$ and BDF methods of order $2 \leq q \leq 5$.

\begin{theorem}
\label{MainTHM-full} 
	Consider the ESFEM--BDF full discretizations \eqref{BDF 2} of the coupled forced mean curvature flow problem \eqref{weak form} and \eqref{eq:surface PDE - alternative weak form}, with \eqref{velocity}, using evolving surface finite elements of polynomial degree~$k \geq 2$ and linearly implicit BDF time discretization of order $q$ with $2 \leq q \leq 5$. 
	Suppose that the forced mean curvature flow problem admits an exact solution $(X,v,\n,H,\u)$ that is sufficiently smooth on the time interval $t\in[0,T]$, and that the flow map $X(\cdot,t):\Gamma^0\to \Gamma(t)\subset\R^3$ is non-degenerate so that $\Gamma(t)$ is a regular surface on the time interval $t\in[0,T]$. 
	Assume that the starting values are sufficiently accurate in the $H^1$ norm at time $t_i=i\tau$ for $i=0,\dots,q-1$.
	
	Then there exist  $h_0 > 0$ and $\tau_0 > 0$ such that for all mesh sizes $h \leq h_0$  and time step sizes $\tau \leq \tau_0$ satisfying the step size restriction 
	\begin{equation} \label{stepsize-restriction}
		\tau \leq C_0 h
	\end{equation}
	(where $C_0>0$ can be chosen arbitrarily),
	the following error bounds for the lifts of the discrete position, velocity, normal vector and mean curvature hold over the exact surface $\Ga(t_n)=\Ga[X(\cdot,t_n)]$ at time $t_n=n\tau\le T$:
	\begin{align*}
	\|(x_h^n)^L - \mathrm{id}_{\Gamma(t_n)}\|_{H^1(\Ga(t_n))^3} \leq &\ C(h^k+\tau^q), \\
	\|(v_h^n)^L - v(\cdot,t_n)\|_{H^1(\Ga(t_n))^3} \leq &\ C(h^k+\tau^q), \\ 
	\|(\n_h^n)^L - \n(\cdot,t_n)\|_{H^1(\Ga(t_n))^3} \leq &\ C(h^k+\tau^q), \\ 
	\|(H_h^n)^L - H(\cdot,t_n)\|_{H^1(\Ga(t_n))} \leq &\ C(h^k+\tau^q), \\
	\|(\u_h^n)^L - \u(\cdot,t_n)\|_{H^1(\Ga(t_n))} \leq &\ C(h^k+\tau^q), \\ 
	\intertext{and also}
	\|(X_h^n)^l - X(\cdot,t_n)\|_{H^1(\Ga_0)^3} \leq &\ C(h^k+\tau^q),
	\end{align*}
	where the constant $C$ is independent of $h$, $\tau$ and $n$ with $n\tau \leq T$, but depends on bounds of higher derivatives of the solution $(X,v,\n,H,\u)$ of the forced mean curvature flow problem, on the length $T$ of the time interval, and on $C_0$.
\end{theorem}

Sufficient regularity assumptions are the following: uniformly in $t\in[0,T]$ and for $j=1,\dotsc,q+1$,
\begin{align*}
&\ X(\cdot,t)  \in  H^{k+1}(\Ga^0)^3, \ \pa_t^{j} X(\cdot,t) \in  H^{1}(\Ga^0)^3  , \\
&\ v(\cdot,t)  \in H^{k+1}(\Ga(t))^3,\ {\mat}^j v(\cdot,t) \in H^{2}(\Ga(t))^3  , \\
\text{for } \ w=(\nu,H,\u) , \quad &\ w(\cdot,t) , \mat w(\cdot,t) \in  W^{k+1,\infty}(\Ga(t))^5, \
{\mat}^j w(\cdot,t) \in  H^{2}(\Ga(t))^5  .
\end{align*}
For the starting values, sufficient approximation conditions are as follows: for $i=0,\dotsc,q-1$,
\begin{align*}
	\|(x_h^i)^L - \mathrm{id}_{\Gamma(t_i)}\|_{H^1(\Ga(t_i))^3} &\leq C(h^k+\tau^q), \qquad \\
	\text{for } \ w=(\nu,H,\u) , \qquad \|(w_h^i)^L - w(\cdot,t_i)\|_{H^1(\Ga(t_i))^5} &\leq C(h^k+\tau^q), \qquad
\end{align*}
and in addition, for $i=1,\dotsc,q-1$,
\begin{equation*}
	\tau^{1/2} \Big\|\frac{1}{\tau}\big(X_h^i - X_h^{i-1} \big)^l - \frac{1}{\tau}\big(X(\cdot,t_i) -  X(\cdot,t_{i-1}) \big) \Big\|_{H^1(\Ga_0)^3} \leq C(h^k+\tau^q) .
\end{equation*}

\bcl
Since \eqref{BDF 2} is the same as the matrix--vector form of mean curvature flow in \cite[equation~(5.1)]{MCF} (recalling that here $\bfw=(\bfn;\bfH;\bfu)$ takes the role of $\bfu=(\bfn;\bfH)$ of \cite{MCF}) and the only problematic additional term in \eqref{BDF 2} is the term 
$\bfA(\widetilde\bfx^n)\widetilde\bfu^n$ that appears in $\bff(\widetilde \bfx^n,\widetilde \bfw^n)$,
the proof of Theorem~\ref{MainTHM-full}  directly follows from the error analysis presented in \cite{MCF} together with the modification concerning $\bfA(\bfx)\bfu$ given in the proof of Theorem~\ref{MainTHM}.
 
%

\begin{remark}
\label{remark:fully discrete convergence}
	\bcl For the second algorithm \eqref{BDF 1}, we expect that \ecl a fully discrete error estimate can be obtained by combining the stability results for the coupled mean curvature flow, \cite[Proposition~10.1]{MCF}, with the extension of the stability analysis for the surface PDE \cite[Proposition~6.1]{KL2018} (via energy estimates  obtained by testing with $\dot{\bfe}^n$).
	We note here that this extension, in particular the analogous steps to part (iv) in \cite[Proposition~10.1]{MCF}, is lengthy and possibly nontrivial. \bcl Numerical experiments presented in Section~\ref{section:numerics} illustrate that optimal-order error estimates are also observed for the scheme \eqref{BDF 1}.\ecl
\end{remark}

\section{Numerical experiments}
\label{section:numerics}

We present numerical experiments for the forced mean curvature flow, using both \eqref{BDF 2} and \eqref{BDF 1}. For our numerical experiments we consider the problem coupling forced mean curvature flow (with a new parameter $\eps>0$) of the surface $\Ga(X(\cdot,t))$, together with evolution equations for its normal vector $\n$ and mean curvature $H$, where the forcing is given through the solution $\u$ of a reaction--diffusion problem on the surface:
\begin{equation}
\label{eq:numerics problem}
	\begin{aligned}
		\mat \u = &\ - \u (\nb_{\Ga[X]} \cdot v) + \varDelta_{\Ga[X]} \u + f(\u,\nb_{\Ga[X]}\u) + \varrho_1 , \\
		v =&\ - \eps H \nu + g\bigl(u) \nu + \varrho_2 , \\
		\mat \nu =&\ \eps \varDelta_{\Ga[X]} \nu + \eps |A|^2 \nu - \nb_{\Ga[X]} (g(\u)) + \varrho_3 , \\
		\mat H =&\ \eps \varDelta_{\Ga[X]} H + \eps |A|^2 H - \varDelta_{\Ga[X]} (g(\u)) - |A|^2 g(\u) + \varrho_4 , \\
		\partial_t X =&\ v ,
	\end{aligned}
\end{equation}
where the inhomogeneities $\varrho_i$ are scalar or vector valued functions on $\R^3 \times [0,T]$, to be specified later on.

We used this problem to perform:
\begin{itemize}
	\item[-] A convergence order experiment for the algorithm \eqref{BDF 2}, in order to illustrate our theoretical results of Theorem~\ref{MainTHM} and \ref{MainTHM-full}.
	\item[-] A convergence order experiment for algorithm \eqref{BDF 1}, illustrating Remark~\ref{remark:semi-discrete convergence} and \ref{remark:fully discrete convergence}.
	\item[-] An experiment, using algorithm \eqref{BDF 2}, for a tumour growth model from \cite[Section~5]{BarreiraElliottMadzvamuse2011}, where one component of a reaction--diffusion surface PDE system forces the mean curvature flow motion of the surface. This experiment allows a direct comparison on the same problem with other methods published in the literature.
\end{itemize}

\bbk All our numerical experiments use quadratic evolving surface finite elements, and linearly implicit BDF methods. The numerical computations were carried out in Matlab. The initial meshes for all surfaces were generated using DistMesh \cite{distmesh}, without taking advantage of the symmetries of the surfaces. \ebk 

\subsection{Convergence experiments}

In order to illustrate the convergence results of Theorem~\ref{MainTHM} and \ref{MainTHM-full}, we have computed the errors between the numerical and exact solutions of the system \eqref{eq:numerics problem}, where the forcing is set to be $g(\u)=\u$, and $\epsilon=1$. The reaction term in the PDE is $F(\u,\nb_{\Ga[X]}\u)=\u^2$. The inhomogeneities $\varrho_i$ are chosen such that the exact solution is $X(q,t)=R(t) q$, with $q$ on the initial surface  $\Gamma_0$, the sphere with radius $R_0$, and $\u(x,t)=e^{-t}x_1x_2$, for all $x \in \Ga[X]$ and $0\leq t \leq T$.
The function $R$ satisfies the logistic differential equation:
\begin{align*}
\frac{\d R\t}{\d t}  =&\ \bigg(1 - \frac{R\t}{R_1}\bigg) R\t, \qquad t \in [0,T], \\
R(0) =&\ R_0 ,
\end{align*}
with $R_1 \geq R_0$, i.e.~the exact evolving surface $\Ga[X(\cdot,t)]$ is a sphere with radius $R\t=R_0R_1 \big( R_0(1-e^{-t}) + R_1 e^{-t}\big)\inv$.

\begin{figure}[htbp]
	\centering
	\includegraphics[width=\textwidth]{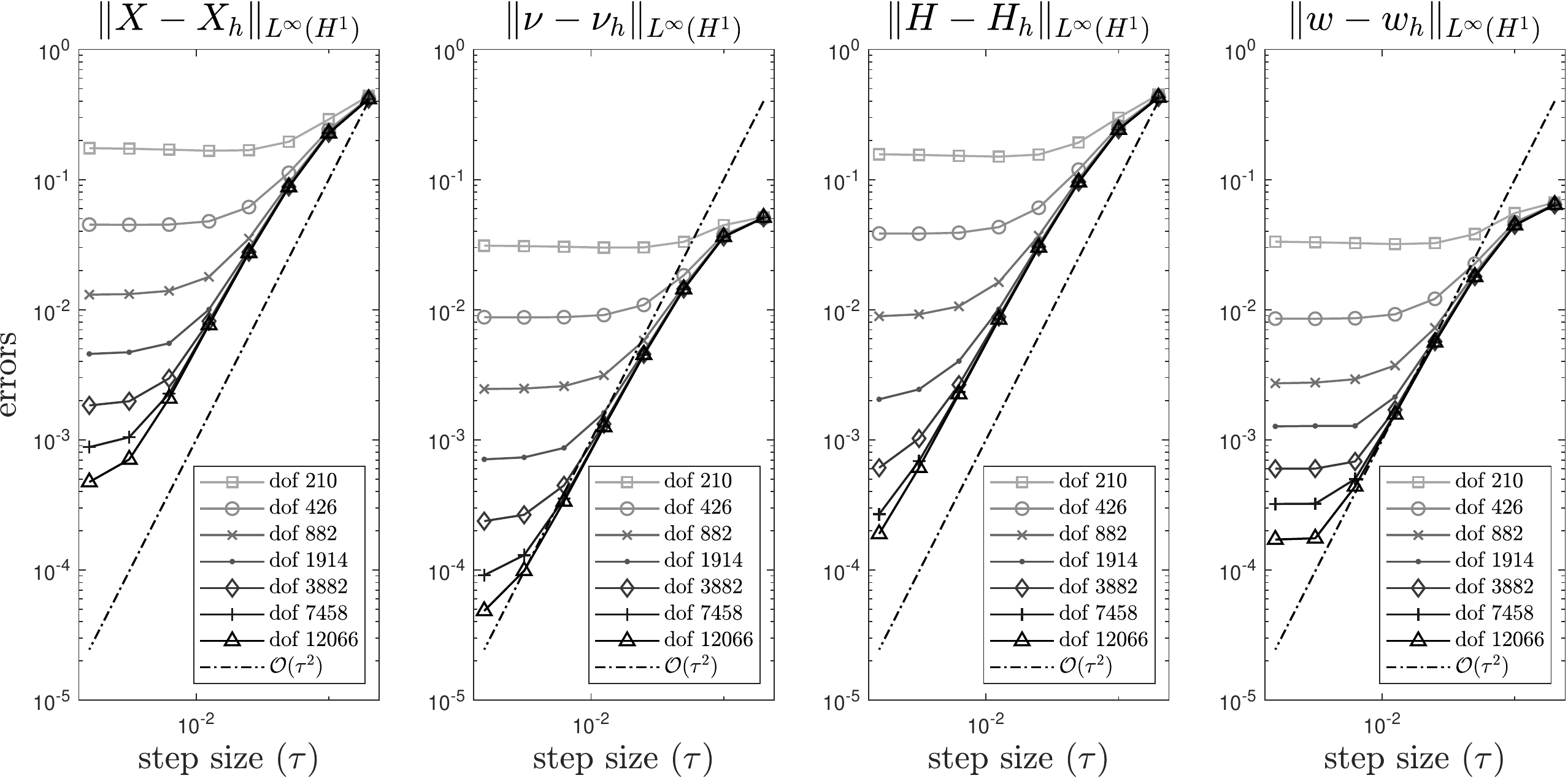}
	\caption{Temporal convergence of the algorithm \eqref{BDF 2} for forced MCF with $g(\u)=\u$, using BDF2 / quadratic ESFEM.}
	\label{fig:conv_time_A}
\end{figure}
\begin{figure}[htbp]
	\centering
	\includegraphics[width=\textwidth]{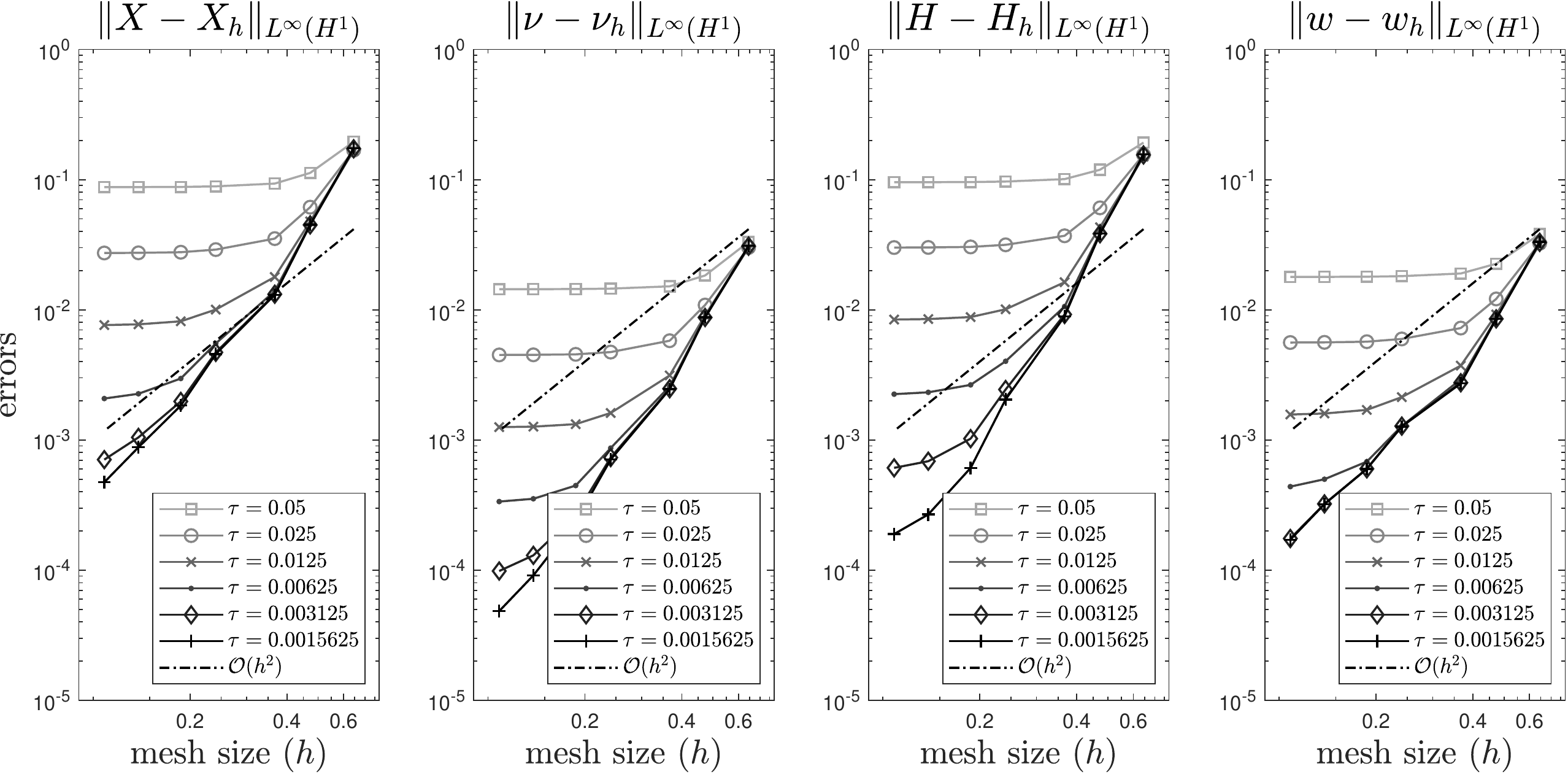}
	\caption{Spatial convergence of the algorithm \eqref{BDF 2} for forced MCF with $g(\u)=\u$, using BDF2 / quadratic ESFEM.}
	\label{fig:conv_space_A}
\end{figure}

Using the algorithm in \eqref{BDF 2} with $2$-step BDF method and quadratic evolving surface FEM, we computed approximations to forced mean curvature flow, using $R_0=1$ and $R_1=2$, until time $T=1$. For our computations we used a sequence of time step sizes $\tau_k=\tau_{k-1}/2$ with $\tau_0 = 0.2$, and a sequence of initial meshes of mesh widths $h_k \approx 2^{-1/2} h_{k-1}$ with $h_0 \approx 0.5$.
The numerical experiments suggest that the step size restriction \eqref{stepsize-restriction} is not required in practice.

In Figure~\ref{fig:conv_time_A} and \ref{fig:conv_space_A} we report the errors between the exact and both numerical solutions for all four variables, i.e.~the surface error, the errors in the dynamic variables $\nu$ and $H$, and the error in the PDE variable $\u$.
The logarithmic plots show the $L^\infty(H^1)$ norm errors against the time step size $\tau$ in Figure~\ref{fig:conv_time_A}, and against the mesh width $h$ in Figure~\ref{fig:conv_space_A}.
The lines marked with different symbols correspond to different mesh refinements and to different time step sizes in Figure~\ref{fig:conv_time_A} and \ref{fig:conv_space_A}, respectively.

In Figure~\ref{fig:conv_time_A} we can observe two regions: a region where the temporal discretization error dominates, matching to the $O(\tau^2)$  order of convergence of our theoretical results, and a region, with small time step sizes, where the spatial discretization error dominates (the error curves flatten out). For Figure~\ref{fig:conv_space_A}, the same description applies, but with reversed roles.

Both the temporal and spatial convergence, as shown by Figures~\ref{fig:conv_time_A} and \ref{fig:conv_space_A}, respectively, are in agreement with the theoretical convergence results of Theorem~\ref{MainTHM} and \ref{MainTHM-full} (note the reference lines).

\medskip
We have performed the same convergence experiments using algorithm \eqref{BDF 1}, which, in view of Remarks~\ref{remark:semi-discrete convergence} and \ref{remark:fully discrete convergence}, and the stability and convergence results of previous works \cite{LubichMansourVenkataraman_bdsurf,KLLP2017,KL2018,MCF}, should also have the same convergence properties as the algorithm \eqref{BDF 1}. As Figures \eqref{fig:conv_time_B} and \eqref{fig:conv_space_B} (created analogously as Figure~\ref{fig:conv_time_A} and \ref{fig:conv_space_A}) illustrate, this expectation appears to be fulfilled.

\begin{figure}[htbp]
	\centering
	\includegraphics[width=\textwidth]{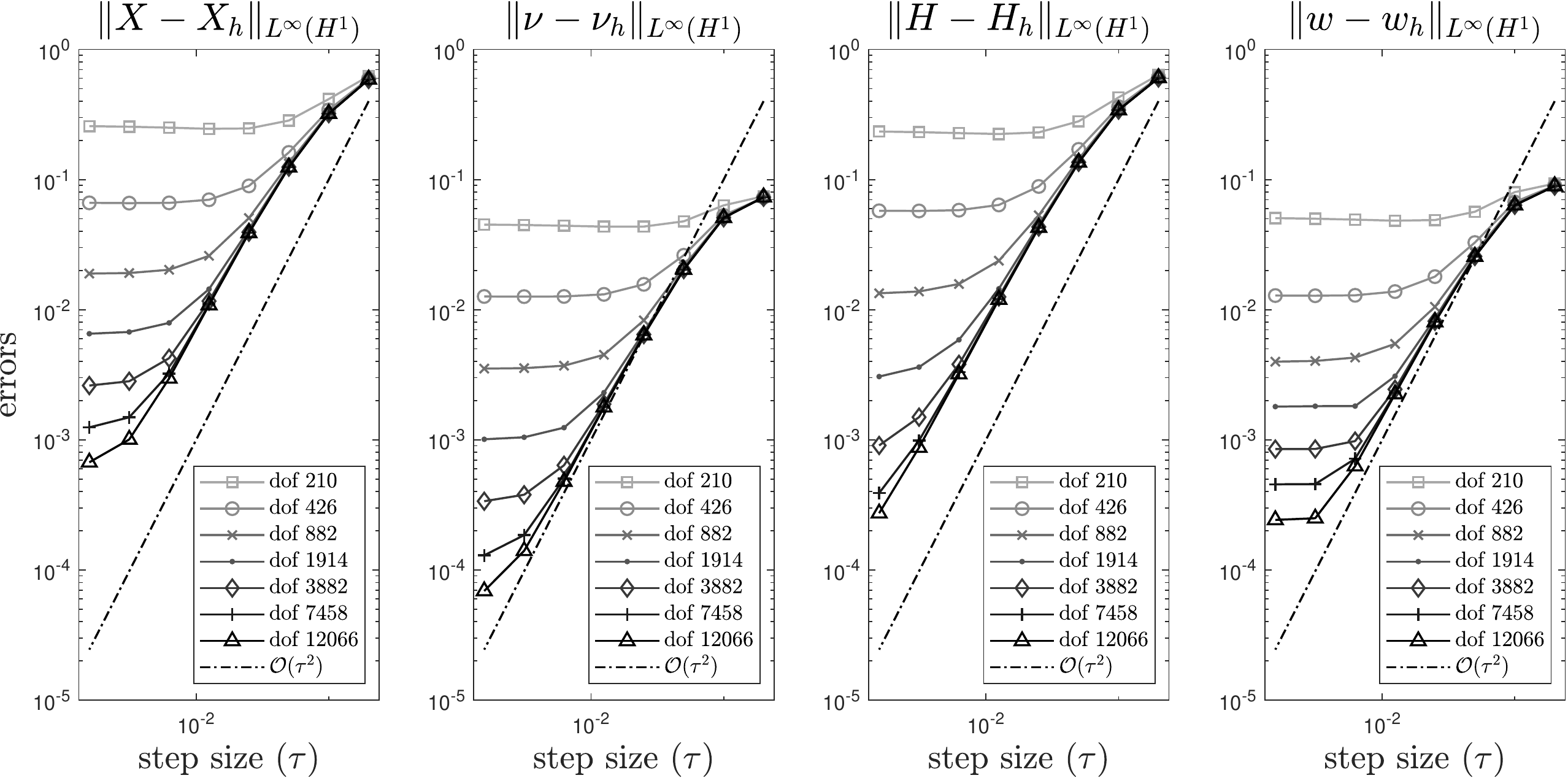}
	\caption{Temporal convergence of the algorithm \eqref{BDF 1} for forced MCF with $g(\u)=\u$, using BDF2 / quadratic ESFEM.}
	\label{fig:conv_time_B}
\end{figure}
\begin{figure}[htbp]
	\centering
	\includegraphics[width=\textwidth]{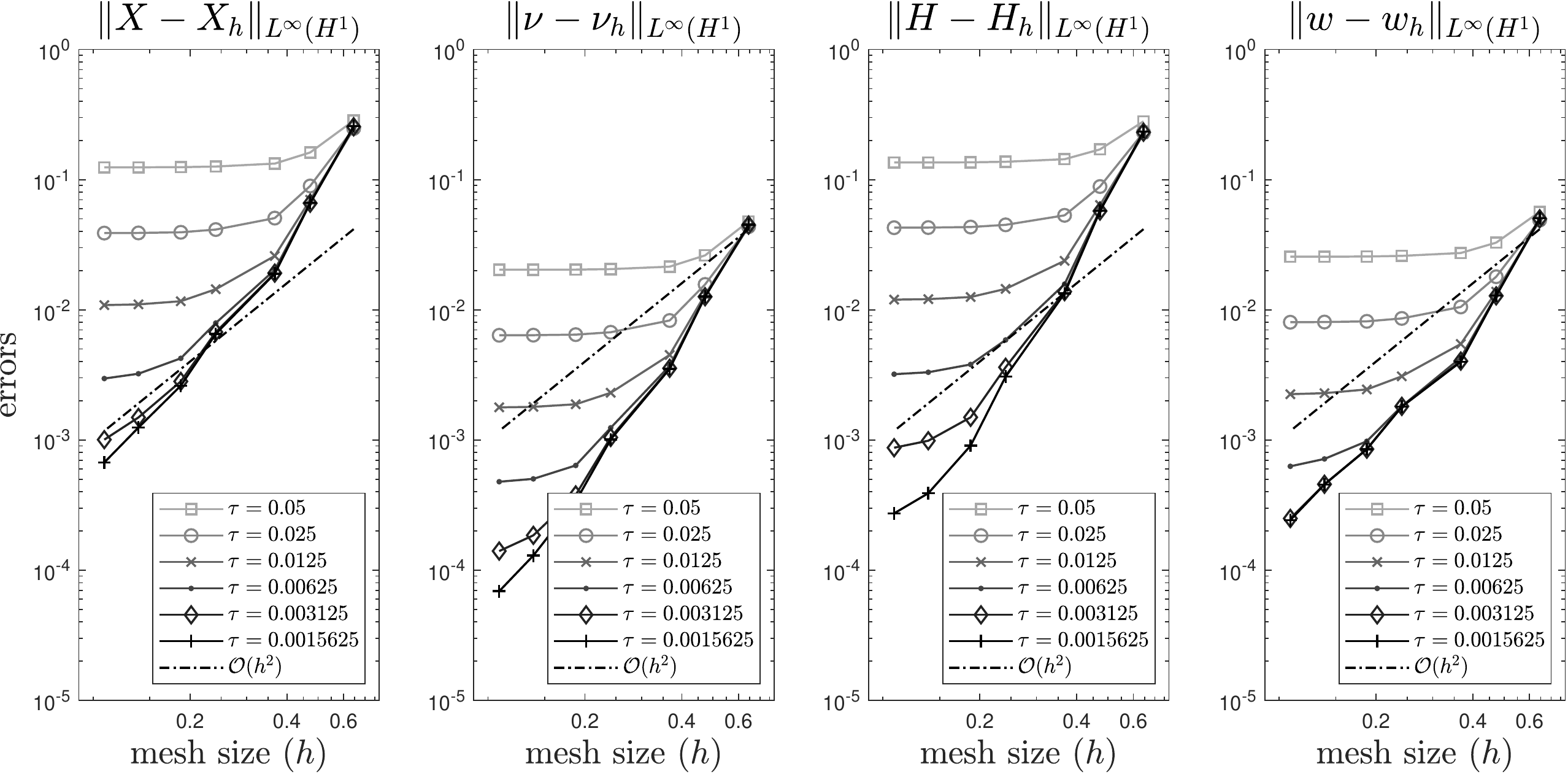}
	\caption{Spatial convergence of the algorithm \eqref{BDF 1} for forced MCF with $g(\u)=\u$, using BDF2 / quadratic ESFEM.}
	\label{fig:conv_space_B}
\end{figure}

\bbk 
We have obtained similar convergence plots for the non-linear forcing term $g(\u)=\tfrac12 \u^2$ for both algorithms.
\ebk

\subsection{Tumour growth}

We performed numerical experiments, using \eqref{BDF 2}, on a well-known model for forced mean curvature flow from \cite[Section~5]{BarreiraElliottMadzvamuse2011}: The problem \eqref{eq:numerics problem}, with vector valued unknown $\u=(\u_1,\u_2)$ and with a small parameter $\eps = 0.01$, models solid tumour growth, for further details we refer to \cite{CrampinGaffneyMaini1999,CrampinGaffneyMaini2002,CGG} and \cite{BarreiraElliottMadzvamuse2011}. Our results can be compared to those in these references, in particularly with those in \cite{BarreiraElliottMadzvamuse2011}.

The surface PDE system for $\u=(\u_1,\u_2)$ describes the activator--depleted kinetics, and has diffusivity constants $1$ and $d=10$ for $\u_1$ and $\u_2$, respectively. The reaction term is given by, with $\gamma>0$,
\begin{equation*}
F(\u) = F(\u_1,\u_2) = \left(\begin{array}{c}
\gamma \big( a - \u_1 + \u_1^2\u_2 \big) \\ 
\gamma \big( b - \u_1^2\u_2 \big) 
\end{array}\right) ,
\end{equation*}
while in the velocity law the non-linearity is given by
\begin{equation*}
g(\u) = g(\u_1,\u_2) = \delta \,\u_1 .
\end{equation*}

The parameters are chosen exactly as in \cite[Table~5]{BarreiraElliottMadzvamuse2011}: $d=10$, $a=0.1$, $b=0.9$, $\delta=0.1$, and $\epsilon=0.01$. The parameter $\gamma$ will be varied for different experiments.

The initial data for all of the presented experiments are obtained (exactly as in \cite[Section~4.1.1 and Figure~8]{BarreiraElliottMadzvamuse2011}) by integrating the reaction--diffusion system on the fixed unit sphere over the time interval $[0,5]$, with small random perturbations of the steady state $\u_1 = a+b$ and $\u_2=b/(a+b)^2$ as initial data. Further initial values (for $i=1,\dotsc,q-1$) for high-order BDF methods are computed using a cascade of steps performed by the corresponding lower order methods.

To mitigate the stiffness of the non-linear term, the linear part of $F(\u)$ is handled fully implicitly, while the non-linear parts of $F$, and the velocity law as well, are treated linearly implicitly using the extrapolation \eqref{eq:extrapolation def}.

In Figure~\ref{fig:sol_gamma30} and \ref{fig:sol_gamma300} we report on the evolution of the surface (and the approximated mean curvature and normal vector) and the component $\u_1$ for parameters $\gamma=30$ and $\gamma=300$, respectively, at different times over the time interval $[5,8]$. \bbk In these plots the linear interpolation of the computed quadratic surface is plotted (since Matlab can only visualise polygonal objects). \ebk 
Figure~\ref{fig:sol_gamma30} and \ref{fig:sol_gamma300} we present the surface evolution and the component $\u_1$ of the surface PDE system (left-hand side columns) and the computed mean curvature $H_h$ and normal vector $\nu_h$ (right-hand side columns) at times $t=5,6,7,8$ (the rows from top to bottom), on a mesh with $3882$ nodes and time step size $\tau=0.0015625$. In particular the top rows show the initial data where the surface evolution is started.
The obtained results for the surface evolution and the reaction--diffusion PDE system (left columns) match nicely (note the random effects in generating initial data) to previously reported results.

\bbk In spite of the smoothing effect of the mean curvature flow, for some more complicated examples it would be beneficial to use an algorithm which allows the tangential motion of the surface nodes, for example based on the DeTurck trick \cite{ElliottFritz_DT}, or on the velocity law $v \cdot \nu = V$, e.g., \cite{BGN2007,BGN2008}, or on ALE techniques \cite{ElliottVenkataraman_ALE,KovacsPower_ALEdiscrete,ALEmap}. However, in our experiments -- both here and in \cite{MCF} -- this was not found necessary. \ebk 

\begin{figure}[htbp]
	\centering
	\includegraphics[trim={80 120 90 90},clip,height=0.24\textheight]{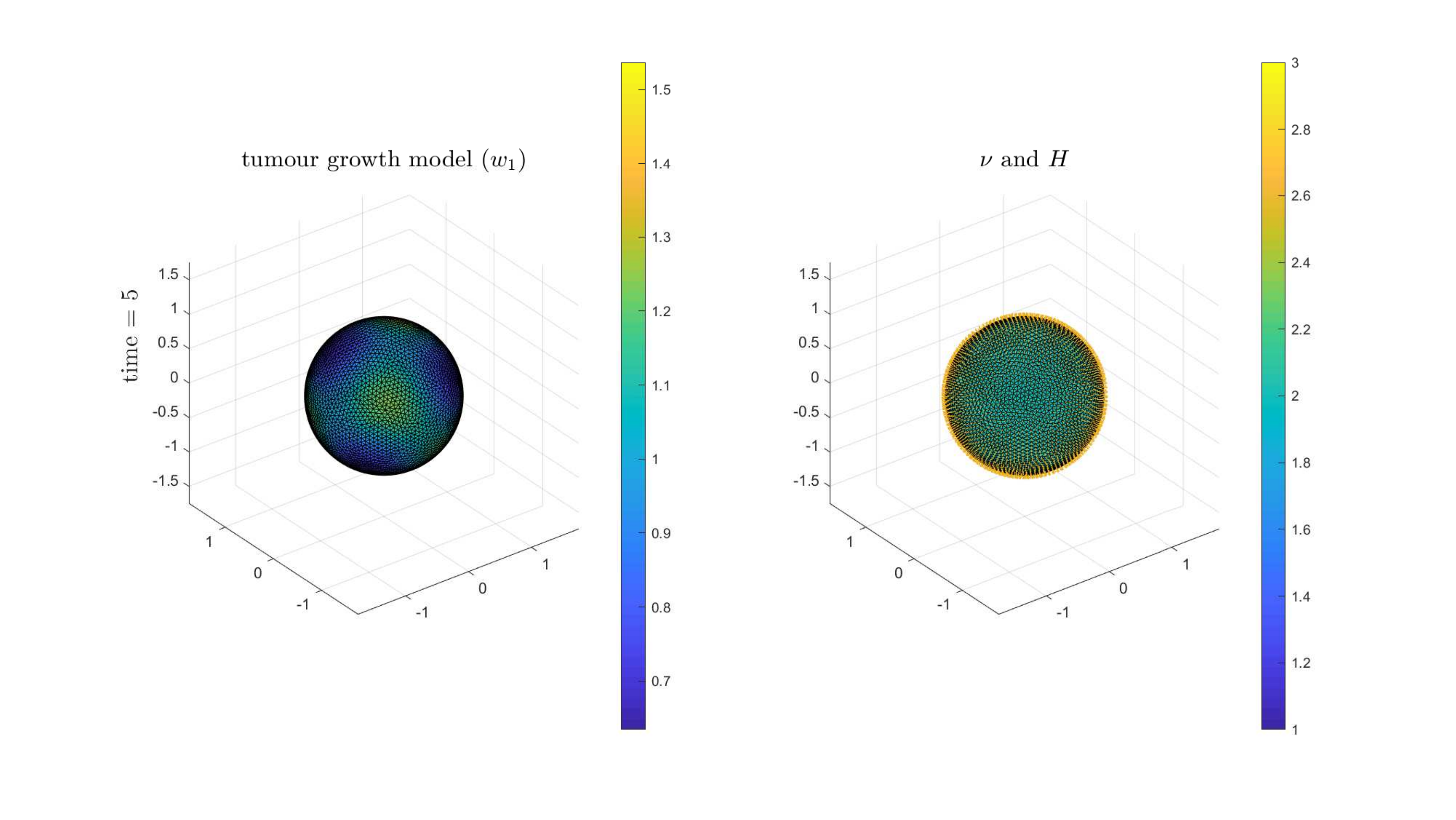}
	\includegraphics[trim={80 120 90 90},clip,height=0.24\textheight]{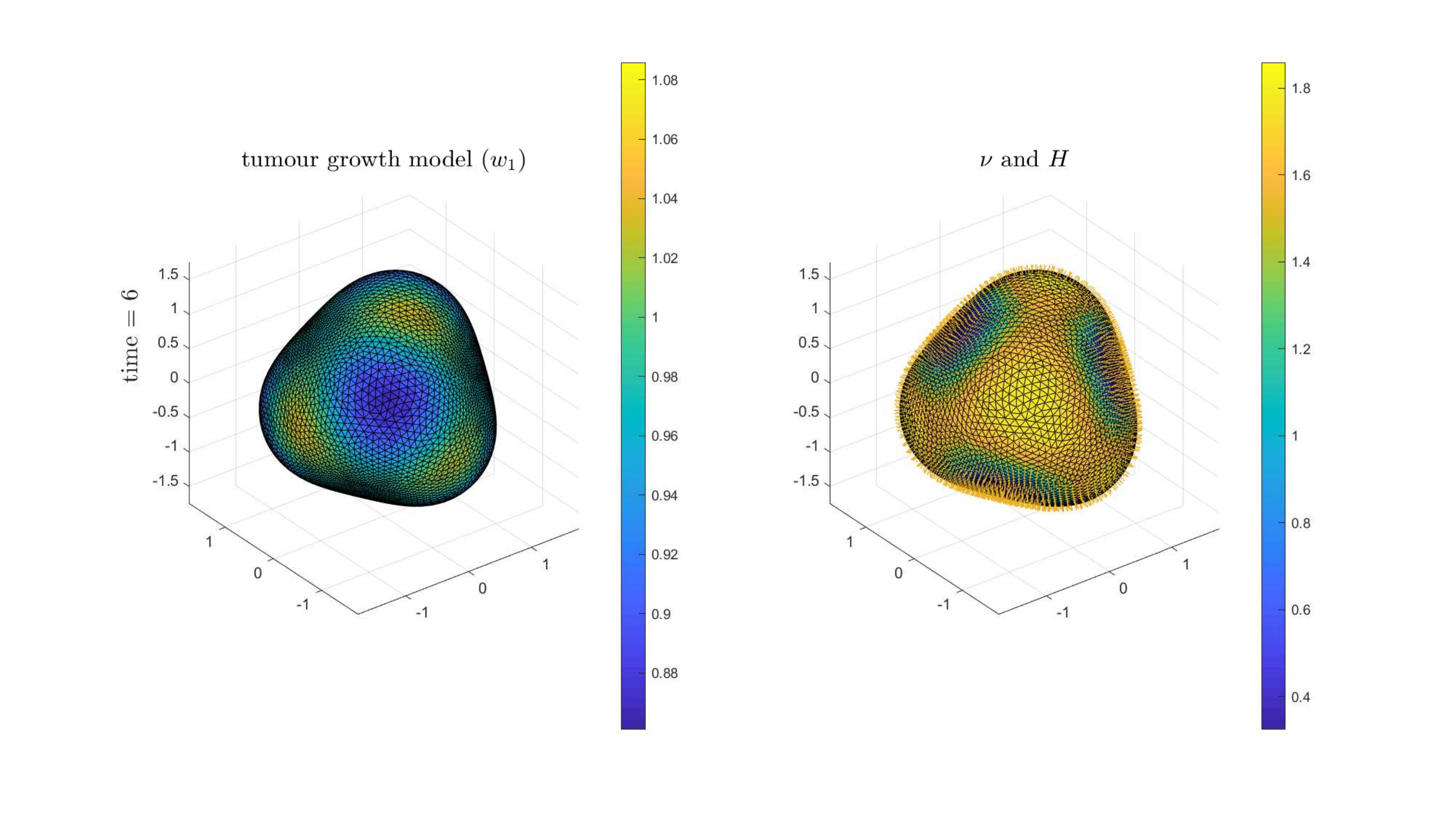}
	\includegraphics[trim={80 120 90 90},clip,height=0.24\textheight]{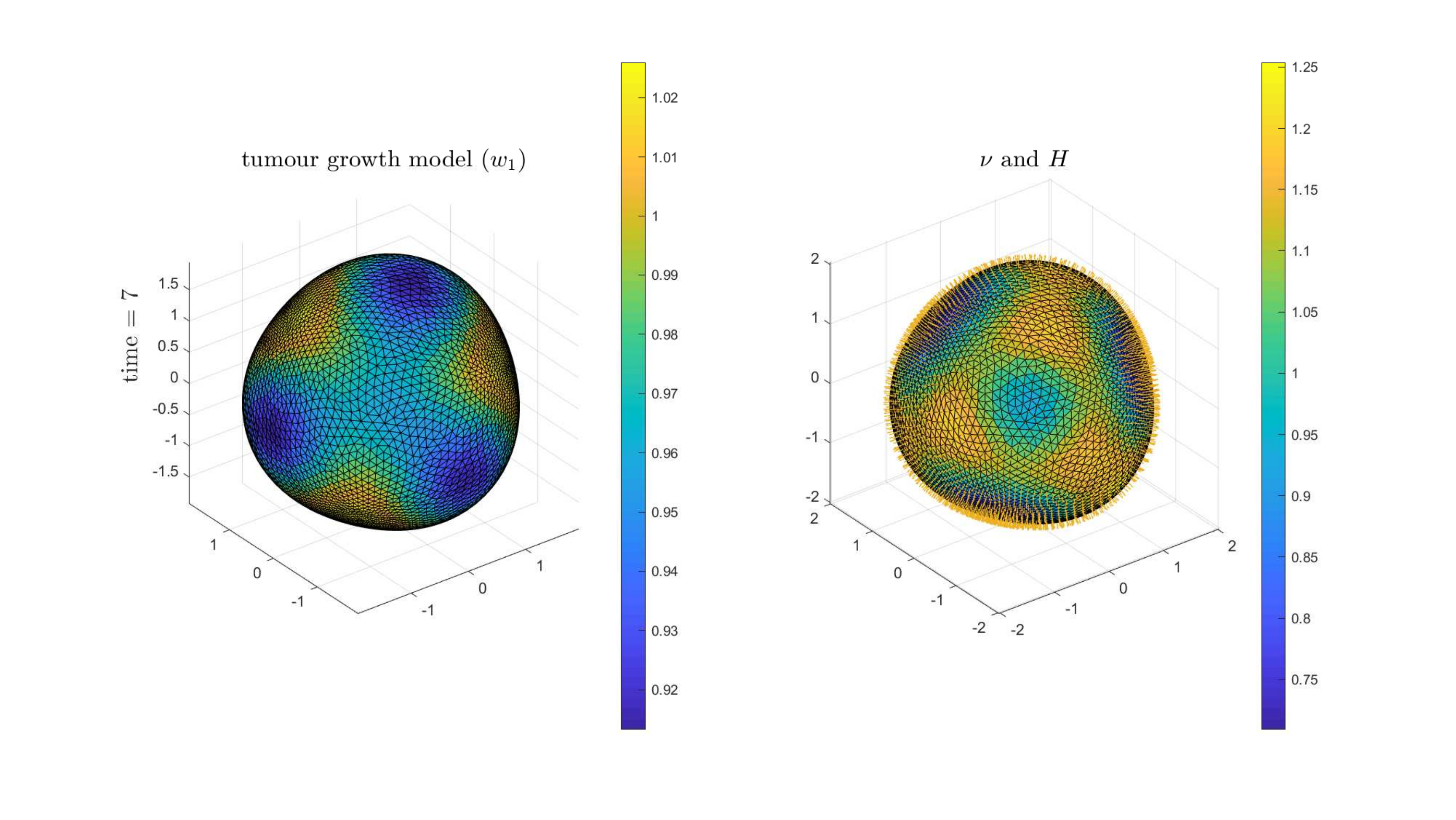}
	\includegraphics[trim={80 120 90 90},clip,height=0.24\textheight]{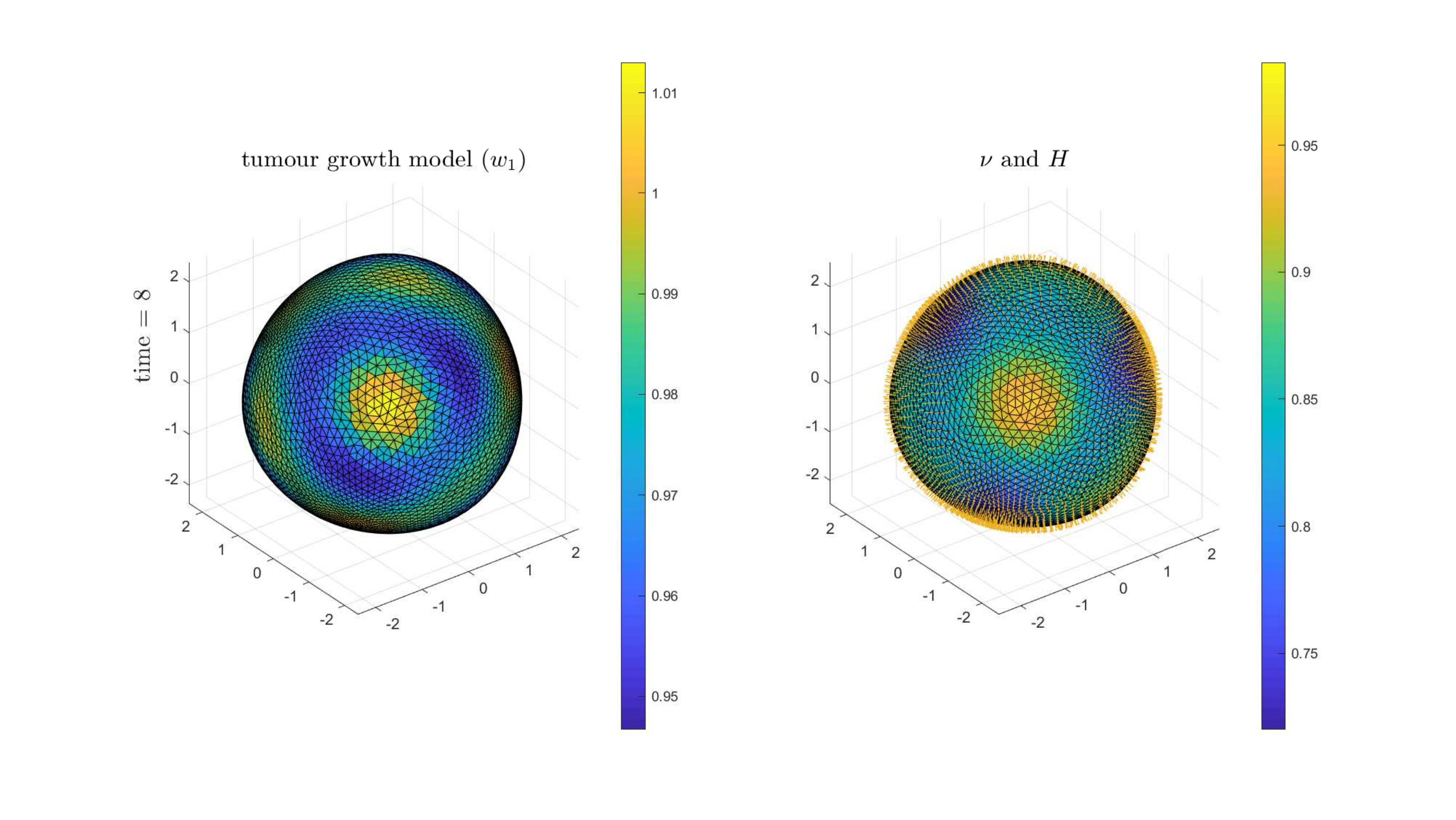}
	\caption{Evolution of the solution ($\u_1$), normal vector and mean curvature for tumour growth model with $\gamma=30$ at time $t=5,6,7,8$; dof $3882$.}
	\label{fig:sol_gamma30}
\end{figure}

\begin{figure}[htbp]
	\centering
	\includegraphics[trim={80 120 90 90},clip,height=0.24\textheight]{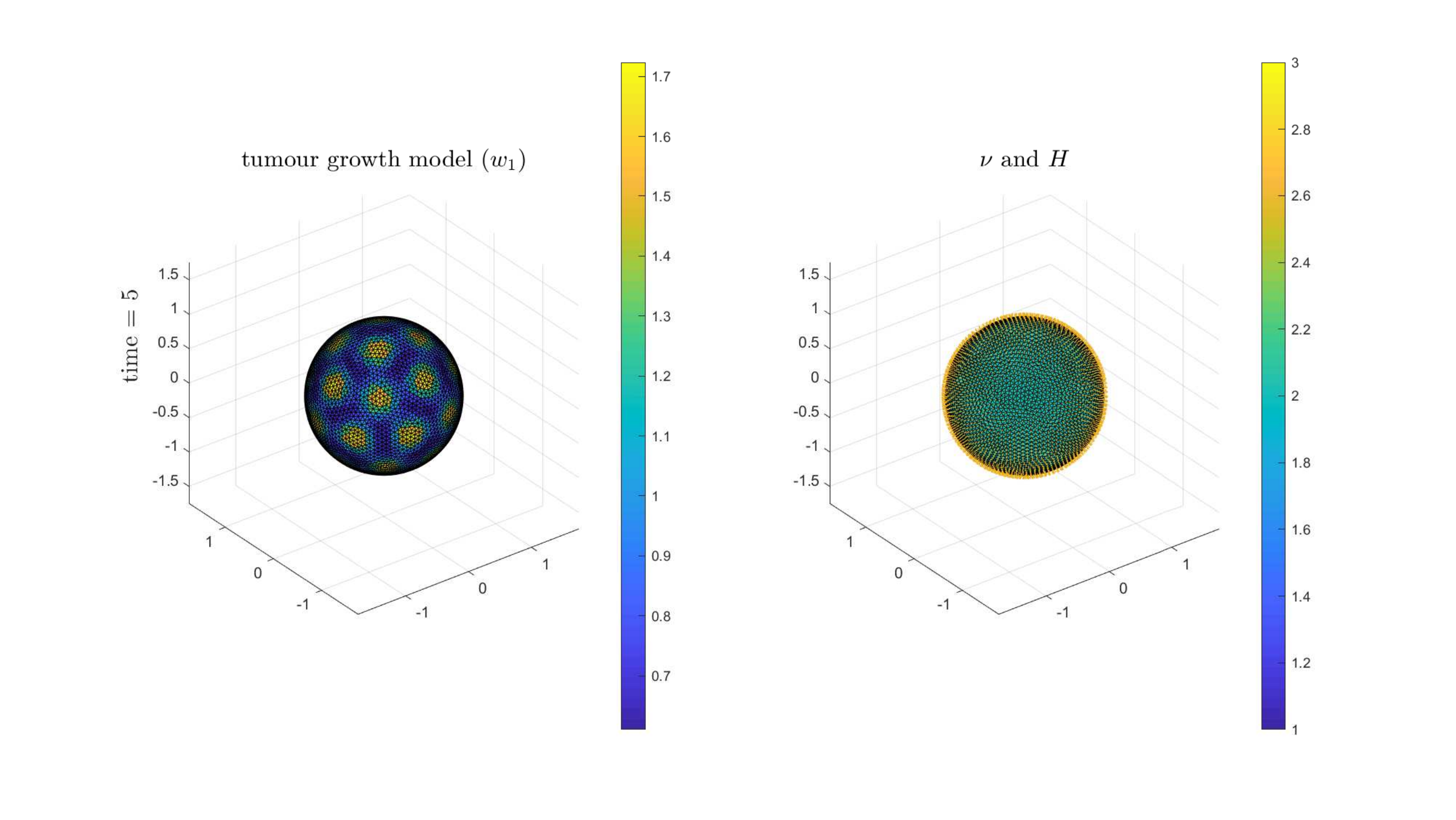}
	\includegraphics[trim={80 120 90 90},clip,height=0.24\textheight]{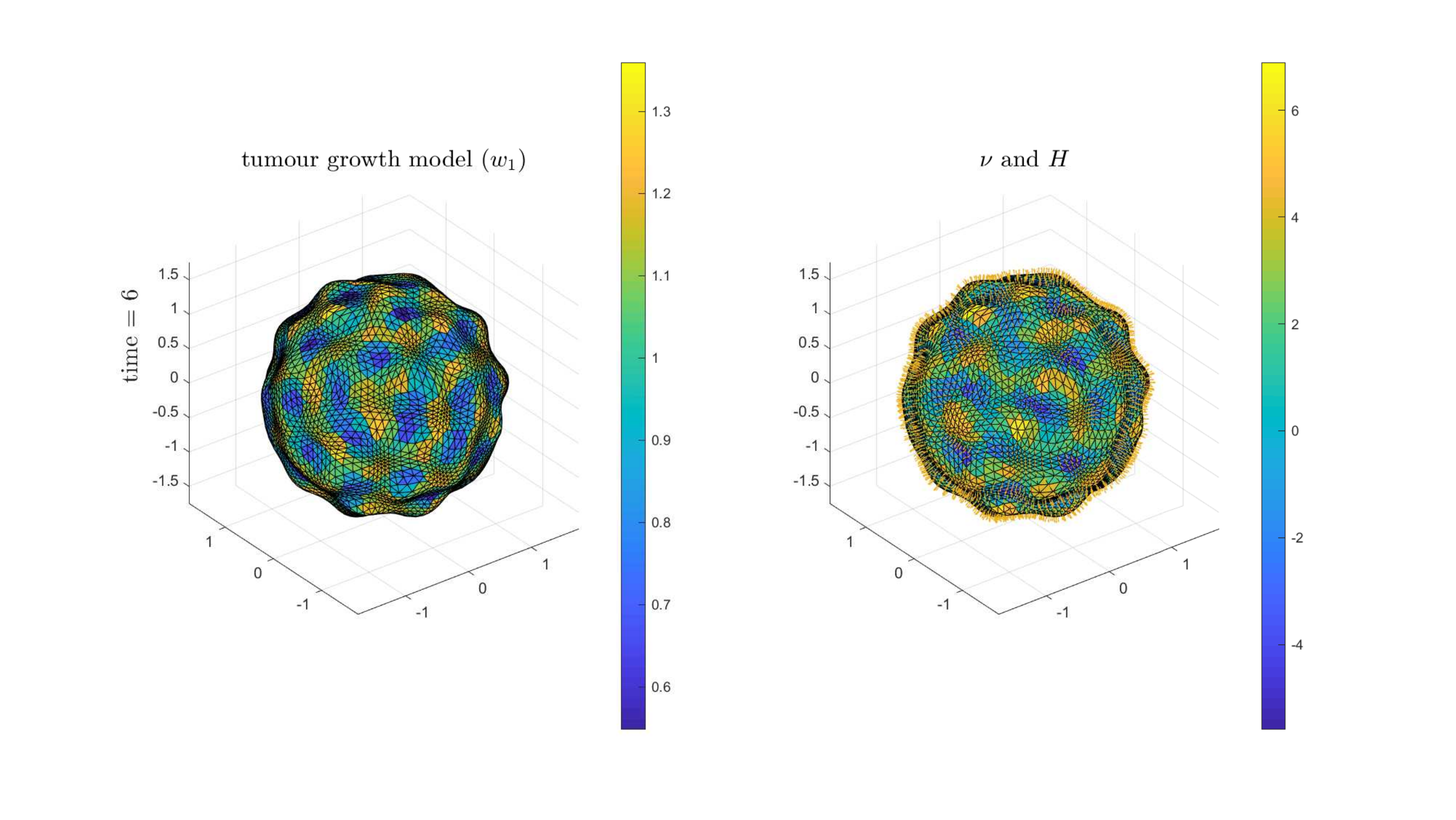}
	\includegraphics[trim={80 120 90 90},clip,height=0.24\textheight]{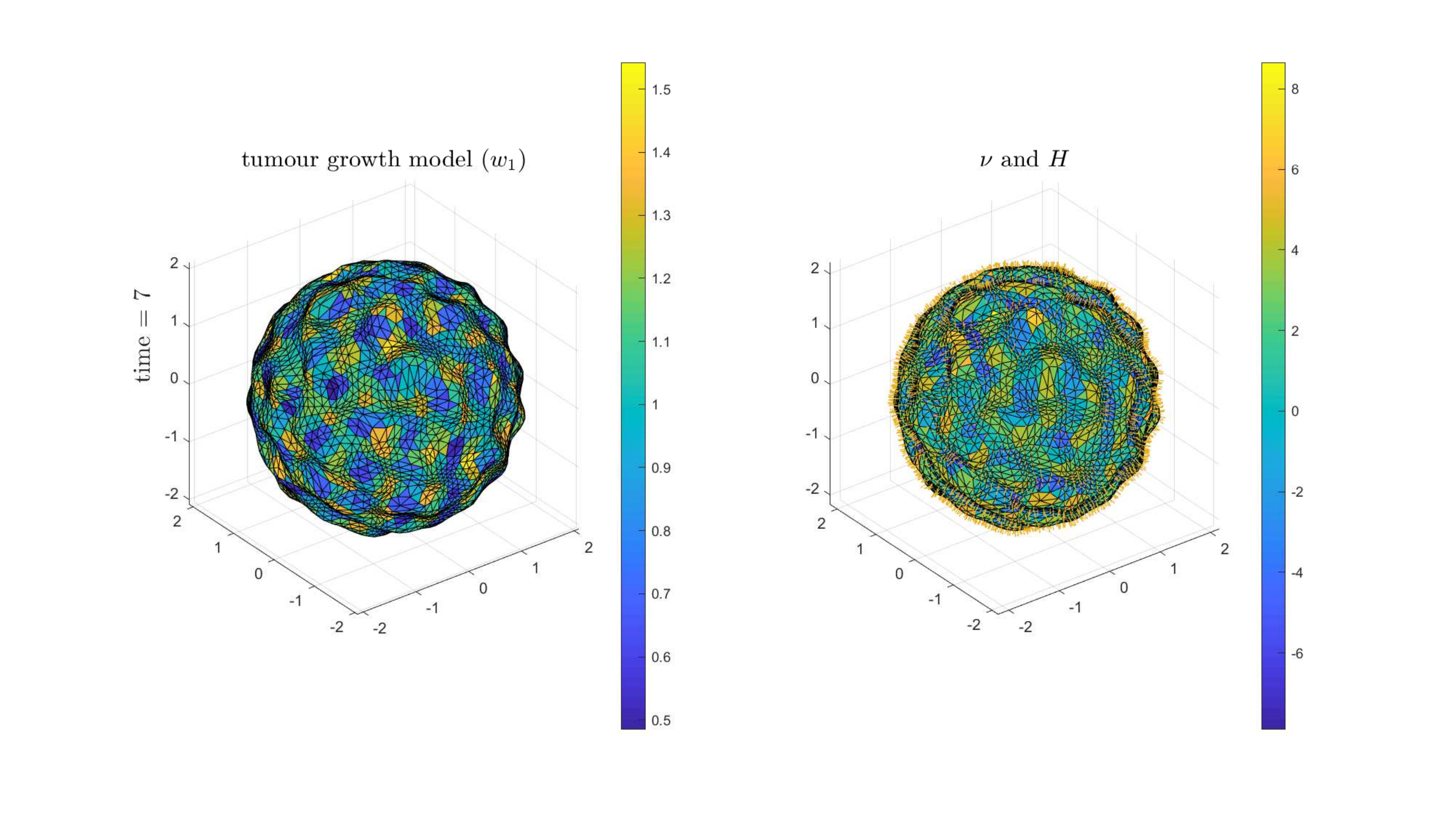}
	\includegraphics[trim={80 120 90 90},clip,height=0.24\textheight]{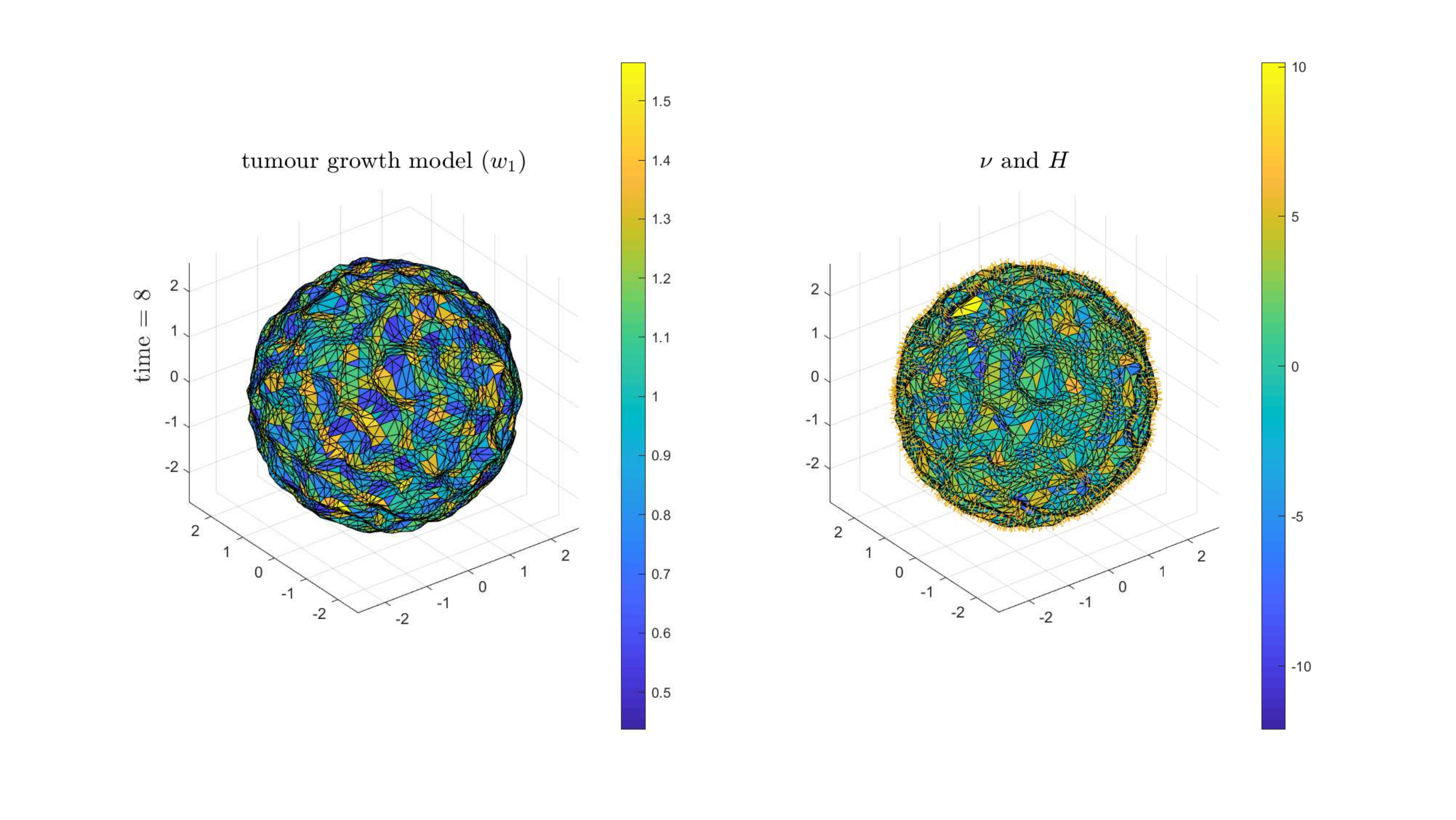}
	\caption{Evolution of the solution ($\u_1$), normal vector and mean curvature for tumour growth model with $\gamma=300$ at time $t=5,6,7,8$; dof $3882$.}
	\label{fig:sol_gamma300}
\end{figure}

\section*{Acknowledgement}
The work of Bal\'azs Kov\'acs and Christian Lubich is supported by Deutsche Forschungsgemeinschaft, SFB 1173. 
The work of Buyang Li is partially supported by an internal grant (Project ZZKQ) of The Hong Kong Polytechnic University.

\bibliographystyle{abbrv}

\end{document}